\documentclass[10pt]{amsart}
\usepackage{amssymb}
\usepackage{bm}
\usepackage{graphicx}
\usepackage[centertags]{amsmath}
\usepackage{amsfonts}
\usepackage{amsthm}
\usepackage{graphicx}
\newtheorem{thm}{Theorem}
\newtheorem{cor}[thm]{Corollary}
\newtheorem{lem}[thm]{Lemma}

\newtheorem{fact}[thm]{Fact}

\newtheorem{defn}[thm]{Definition}
\theoremstyle{definition}

\newcommand{\con}{\smallfrown}
\newcommand{\meg}{\geqslant}
\newcommand{\mik}{\leqslant}
\newcommand{\mikw}{\leqslant^\mathrm{w}}

\newcommand{\pred}{\mathrm{Pred}}
\newcommand{\suc}{\mathrm{Succ}}
\newcommand{\imsuc}{\mathrm{ImmSuc}}
\newcommand{\ws}{\mathrm{ws}}
\newcommand{\w}{\mathrm{W}}
\newcommand{\xx}{\mathbf{x}}
\newcommand{\yy}{\mathbf{y}}
\newcommand{\ttt}{\mathbf{t}}
\newcommand{\sss}{\mathbf{s}}
\newcommand{\SSS}{\mathbf{S}}
\newcommand{\tp}{\mathrm{top}}
\newcommand{\bt}{\mathrm{bot}}

\begin{document}
\title{A disjoint union theorem for trees}
\author{Stevo Todorcevic}
\author{Konstantinos Tyros}

\address{Department of Mathematics, University of Toronto, Toronto, Canada, M5S 2E4. Institut de Math\'ematiques de Jussieu, UMR 7586, 2 pl. Jussieu, case 7012, 75251 Paris Cedex 05, France. }
\email{stevo@math.toronto.edu, stevo.todorcevic@imj-prg.fr}
\address{and}
\address{Department of Mathematics, University of Toronto, Toronto, Canada, M5S 2E4 }
\email{ktyros@math.toronto.edu}

\thanks{Partially suported by grants from NSERC and CNRS}

\maketitle

\begin{abstract}
  We prove an infinitary disjoint union theorem for level products of trees. To implement the proof we develop a Hales--Jewett type result for words indexed by a level product of trees.
\end{abstract}

\section{Introduction}

The finite disjoint union theorem, also known as Folkman's theorem (see \cite{GRS}; p.82), is the following statement.

\begin{thm}[ Folkman]
For every pair of positive integers $k$ and $c$ there is an integer $F=F(k,c)$ such that for every $c$-coloring of the power-set
$\mathcal{P}(X)$ of some set $X$ of cardinality $\geq F,$ there is a family $\mathcal{D}=(D_i)_{i=1}^k$ of pairwise disjoint nonempty subsets of $X$
such that the family
\[\Big\{\bigcup_{i\in I} D_i: \emptyset\neq I\subseteq \{1,2,...,k\}\Big\}\]
of unions is monochromatic.
\end{thm}

It is well known that the infinite form of this result must have some restriction on colorings.
In fact, an essentially optimal restriction is given by the following result of T. J. Carlson and S. G. Simpson  given  in \cite{CS} as a consequence of their infinite version of the dual Ramsey theorem stating that for every Suslin measurable (see Section \ref{section_background_material} for definition) finite coloring of the set $\mathcal{E}^\infty$ of all partitions of the natural numbers into infinitely many parts, there exists a partition $X$ in $\mathcal{E}^\infty$ such that the set of all partitions coarser that $X$ is monochromatic. This has the following consequence
\begin{thm}[Carlson-Simpson]
  \label{CS}
  For every Suslin measurable finite coloring of the powerset of the natural numbers, there is a
  sequence $(X_n)_n$ of pairwise disjoint subsets of the natural numbers such that the set
  \[\Big\{\bigcup_{n\in Y}X_n:\;Y\;\text{is a non-empty subset of the natural numbers}\Big\}\]
  is monochromatic.
\end{thm}
In this paper, we view this as a dual form of the classical pigeonhole principle for finite partitions of $\omega$ and we extend it by replacing $\omega$ by any other rooted finitelly branching tree of height $\omega$ with no terminal nodes.
Recall that a most famous pigeonhole principle for trees is the subject of
a deep result due to J. D. Halpern and H. L\"{a}uchli \cite{HL}. In this paper we establish
a dual version of the Halpern--L\"{a}uchli theorem. We prove a result similar to Theorem \ref{CS} where the underling structure is the level product $\otimes\mathbf{T}$ of an infinite vector tree $\mathbf{T}$, that is, a finite sequence of finitely brunching and rooted trees of height $\omega$ with no maximal nodes, instead of the set of the natural numbers (see Section \ref{section_notation_for_trees} for the relevant notation). In particular, we consider the
following subset of the powerset of $\otimes\mathbf{T}$,
\[\mathcal{U}(\mathbf{T})=\{U\subseteq\otimes\mathbf{T}:\;U\;\text{has a minimum}\}.\]
The collection $\mathcal{U}(\mathbf{T})$ can be viewed as a subset of the set $\{0,1\}^{\otimes\mathbf{T}}$ of all
functions from $\otimes\mathbf{T}$ into $\{0,1\}$. We endow the set $\{0,1\}^{\otimes\mathbf{T}}$
with the product topology of the discrete topology on $\{0,1\}$. Then
$\mathcal{U}(\mathbf{T})$ forms a closed subset of $\{0,1\}^{\otimes\mathbf{T}}$.
The subspaces that we consider in this setting are families of pairwise disjoint elements of $\mathcal{U}(\mathbf{T})$ indexed by the level product of a vector subset (we advise the unfamiliar reader with the notions of vector tree, vector subset and dense vector subset to first take a look at Section \ref{section_notation_for_trees}). More precisely, let
$\mathbf{D}$ be a vector subset of $\mathbf{T}$. A $\mathbf{D}$-subspace of $\mathcal{U}(\mathbf{T})$ is a family $\mathbf{U}=(U_\ttt)_{\ttt\in\otimes\mathbf{D}}$ consisting of pairwise disjoint elements from $\mathcal{U}(\mathbf{T})$ such that $\min U_\ttt=\ttt$ for all $\ttt$ in $\otimes\mathbf{D}$. We say that $\mathbf{U}$ is a subspace of $\mathcal{U}(\mathbf{T})$ if it is a $\mathbf{D}$-subspace for some vector subset $\mathbf{D}$ of $\mathbf{T}$, which we denote by $\mathbf{D}(\mathbf{U})$.
For a subspace $\mathbf{U}=(U_\ttt)_{\ttt\in\otimes\mathbf{D}(\mathbf{U})}$
we define its span by the rule
\[[\mathbf{U}]=\Big\{\bigcup_{\ttt\in\Gamma}U_\ttt:\;\Gamma\subseteq\otimes\mathbf{D}(\mathbf{U})\Big\}\cap\mathcal{U}(\mathbf{T}).\]
If $\mathbf{U}$ and $\mathbf{U}'$ are two subspaces of $\mathcal{U}(\mathbf{T})$, we say that $\mathbf{U}'$ is a further subspace of $\mathbf{U}$, we write $\mathbf{U}'\mik\mathbf{U}$, if $[\mathbf{U}']$ is a subset of $[\mathbf{U}]$. Observe that, in particular, this implies that $\mathbf{D}(\mathbf{U}')$ is a vector subset of $\mathbf{D}(\mathbf{U})$.
The main result of this work is the following.
\begin{thm}
  \label{disjoint_Union_tree}
  Let $\mathbf{T}$ be a vector tree of infinite height and $\mathcal{P}$ a Souslin measurable subset of $\mathcal{U}(\mathbf{T})$. Also let $\mathbf{D}$ be a dense 
  vector subset of $\mathbf{T}$ and $\mathbf{U}$ a $\mathbf{D}$-subspace of $\mathcal{U}(\mathbf{T})$.
  Then there exists a subspace $\mathbf{U}'$ of $\mathcal{U}(\mathbf{T})$ with $\mathbf{U}'\mik\mathbf{U}$ such that
  either
  \begin{enumerate}
    \item[(i)]  $[\mathbf{U}']$ is a subset of $\mathcal{P}$ and $\mathbf{D}(\mathbf{U}')$ is a dense 
        vector subset of $\mathbf{T}$, or
    \item[(ii)] $[\mathbf{U}']$ is a subset of $\mathcal{P}^c$ and $\mathbf{D}(\mathbf{U}')$ is a $\ttt$-dense 
        vector subset of $\mathbf{T}$ for some $\ttt$ in $\otimes\mathbf{T}$.
  \end{enumerate}
\end{thm}

 For the proof of Theorem \ref{disjoint_Union_tree} we prove an analogue of the infinite dimensional version of the Hales--Jewett Theorem \cite{C,FK} for maps defined on the level product of a vector tree (see Theorem \ref{tree_Hales_Jewett_infinite} below). We obtain the latter by an application of the Abstract Ramsey Theory developed in \cite{To}. The pigeonhole needed here
 is an analogue of the Hales--Jewett Theorem \cite{HJ} for maps defined on the level product of a vector tree (see Theorem \ref{tree_Hales_Jewett}
  below).

In Section \ref{section_notation_for_trees} we include the notation concerning trees and product of trees, while in Section \ref{section_background_material} we include some background material needed for the present work. Sections \ref{section_notation_words}--\ref{section_proof_tree_Hales_Jewett} are devoted to the proof of Theorem \ref{tree_Hales_Jewett}. The proof of Theorem \ref{tree_Hales_Jewett} follows similar lines to the ones in \cite{K}. In section \ref{section_infinite_Hales_Jewett} we obtain the infinite dimensional version of Theorem \ref{tree_Hales_Jewett} while in Section \ref{section_consequences} we obtain some consequences needed for the proof of the main result
which is the subject of Section \ref{section_main_res}.

\section{Basic notation for trees and vector trees}\label{section_notation_for_trees}
By $\omega$ we denote the set of all non-negative integers, while $\emptyset$ stands for the empty set as well as the empty function. We adopt the convection that $n<\omega$ for all integers $n$.
Moreover, for every two finite sequences $\mathbf{a}$ and $\mathbf{b}$, by $\mathbf{a}^\con\mathbf{b}$ we denote the concatenation of $\mathbf{a}$ and $\mathbf{b}$.
Finally, for a map $f$ and a subset $D$ of the domain of $f$, by $f\upharpoonright D$ we denote the restriction of $f$ on $D$.
\subsection{Trees}
By the term tree we mean a partial ordered set $(T,\mik_T)$ such that
\begin{enumerate}
  \item[(i)] $T$ has a minimum, which is called the root of $T$ and is denoted by $r_T$,
  \item[(ii)] for every $t$ in $T$, the set of the predecessors of $t$ inside $T$
            \[\pred_T(t)=\{s\in T: s<_Tt\},\]
            is well ordered and finite,
  \item[(iii)] for all non-maximal $t\in T$, the set of the immediate successors of $t$ inside $T$
            \[\begin{split}\imsuc_T(t)=\{s\in T:\;& t<_Ts\;\text{and for all}\;s'\in T\\
            &\text{with}\;t<_Ts'\mik_Ts,\;\text{we have that}\;s=s'\}
            \end{split}\]
            is finite and
  \item[(iv)] the tree $T$ is balanced, i.e., all its maximal chains have the same cardinality.
\end{enumerate}

Let us fix a tree $T$. The height of $T$, denoted by $h(T)$, is defined to be the common cardinality of its chains. Let us observe that $T$ is infinite if and only if $h(T)=\omega$.
For every node $t$ in $T$, we set
\[\suc_T(t)=\{t'\in T:\;t\mik_T t'\}\]
and $\ell_T(t)$ to be the cardinality of the set $\pred_T(t)$.
For a non-negative integer $n$ with $n< h(T)$, we define the $n$-th level $T(n)$ of $T$ to be the set of all nodes $t$ in $T$ with
$\ell_T(t)=n$. For a subset $D$ of $T$, we define the level set of $D$ inside $T$ to be the set
\[L_T(D)=\{n\in\omega:\;n<h(T)\;\text{and}\;T(n)\cap D\neq\emptyset\}.\]
A nonempty subset $D$ of $T$ is called a level subset of $T$ if there exists a non-negative integer $n$ with $n<h(T)$ such that $D$ is a subset of $T(n)$. If $D_1, D_2$ are two level subsets of $T$, we say that $D_1$ dominates $D_2$ if for every $t\in D_2$ there exists $s\in D_1$ such that $t\mik_T s$.
If $T$ is infinite, then we say that a subset $D$ of $T$ is dense 
in $T$, if setting $L_T(D)=\{\ell_0<\ell_1<...\}$, we have that for every $k$ in $\omega$ the set $D\cap T(\ell_k)$ dominates $T(k)$. Moreover, if $T$ is infinite and $t$ is a node in $T$, then we say that a subset $D$ of $T$ is $t$-dense 
if $D\cap\suc_T(t)$ is dense 
in $\suc_T(t)$.
Finally, for every pair of non-negative integers $m,n$ with $m\mik n\mik h(T)$, we set \[T\upharpoonright[m,n)=\bigcup_{p=m}^{n-1}T(p)\]
if $m<n$ and $\emptyset$ otherwise. For notational simplicity, for every non-negative integer $n$ with $n\mik h(T)$, by $T\upharpoonright n$, we denote the set $T\upharpoonright[0,n)$.

Assume that $T$ is of infinite height and let $D$ be a dense 
subset of $T$. Set $L_T(D)=\{\ell_0<\ell_1<...\}$.
For every non-negative integer $n$ we set
\[D(n)=D\cap T(\ell_n).\]
For every pair of non-negative integers $m,n$ with $m\mik n$ we set
\[D\upharpoonright[m,n)=\bigcup_{i=m}^{n-1}D(i),\]
while for simplicity we set $D\upharpoonright n=D\upharpoonright[0,n)$.



\subsection{Vector trees and level products}

A vector tree $\mathbf{T}$ is a finite sequence of trees having the same height.
The height of a vector tree $\mathbf{T}$, denoted by $h(\mathbf{T})$, is defined to be the
common height of its components.
Let $d$ be a positive integer and $\mathbf{T}=(T_1,...,T_d)$ be a vector tree. We define the level product of $\mathbf{T}$ as
\[\otimes\mathbf{T}=\bigcup_{n<h(\mathbf{T})} T_1(n)\times...\times T_d(n).\]
We endow the level product $\otimes\mathbf{T}$ with an ordering $\mik_{\otimes\mathbf{T}}$ defined as follows.
For every $\ttt=(t_1,...,t_d)$ and $\sss=(s_1,...,s_d)$ in $\otimes\mathbf{T}$ we set
$\ttt\mik_{\otimes\mathbf{T}}\sss$ if $t_i\mik_{T_i}s_i$ for all $1\mik i\mik d$. It is easy to see that
$(\otimes\mathbf{T},\mik_{\otimes\mathbf{T}})$ is a tree.

A sequence $\mathbf{D}=(D_1,...,D_d)$ is called a vector subset of $\mathbf{T}$ if $D_i\subseteq T_i$ for all $1\mik i\mik d$
and for every $1\mik i,j\mik d$ we have $L_{T_i}(D_i)=L_{T_j}(D_j)$. We set the level set $L_\mathbf{T}(\mathbf{D})$ of $\mathbf{D}$ inside $\mathbf{T}$ to be the set $L_{T_1}(D_1)$.
Moreover, if $\mathbf{D}$ is a vector subset of $\mathbf{T}$, the level product of $\mathbf{D}$ is defined as follows
\[\otimes\mathbf{D}=\bigcup_{n<h(\mathbf{T})}\big(D_1\cap T_1(n)\big)\times...\times \big(D_d\cap T_d(n)\big).\]
We say that a vector subset $\mathbf{D}$ of $\mathbf{T}$
is a vector level subset, if $L_\mathbf{T}(\mathbf{D})$ is a singleton. If $\mathbf{D}_1=(D^1_1,...,D^1_d)$ and $\mathbf{D}_2=(D_1^2,...,D_d^2)$ are two vector level subsets of $\mathbf{T}$, we say that $\mathbf{D}_1$ dominates $\mathbf{D}_2$, if $D_i^1$ dominates $D_i^2$ for all $1\mik i\mik d$.
Moreover, if $\mathbf{T}$ is of infinite height, we say that a vector subset $\mathbf{D}=(D_1,...,D_d)$ is dense 
if $D_i$ is dense 
in $T_i$ for every $1\mik i\mik d$. Finally, if $\mathbf{T}$ is of infinite height and $\ttt=(t_1,...,t_d)$ is in $\otimes\mathbf{T}$, we say that a vector subset $\mathbf{D}=(D_1,...,D_d)$ is $\mathbf{t}$-dense 
if $D_i$ is a $t_i$-dense 
subset of $T_i$ for every $1\mik i\mik d$.

Assume that $\mathbf{T}$ is of infinite height and let $\mathbf{D}=(D_1,...,D_d)$ be a dense 
vector subset of $\mathbf{T}$.
For every non-negative integer $n$ we set
\[\mathbf{D}(n)=\big(D_1(n),...,D_d(n)\big).\]
For every pair of non-negative integers $m,n$ with $m\mik n$ we set
\[\mathbf{D}\upharpoonright[m,n)=\Big(\bigcup_{i=m}^{n-1}D_1(i),...,\bigcup_{i=m}^{n-1}D_d(i)\Big),\]
while for simplicity we set $\mathbf{D}\upharpoonright n=\mathbf{D}\upharpoonright[0,n)$.

\section{Background material}\label{section_background_material}
In this section we gather some background material needed for the present work.
\subsection{The Hales--Jewett Theorem}
The Hales--Jewett Theorem is one of the cornerstones of Ramsey theory. Let us start with some pieces of notation.
Let $N$ be a positive integer and $\Lambda$ a finite alphabet. We view the elements of the set  $\Lambda^N$ of all functions from the set $\{0,...,N-1\}$ into $\Lambda$ as constant words of length $N$ over the alphabet $\Lambda$. Also, let a symbol $v\not\in\Lambda$. A variable word $w(v)$ of length $N$ over $\Lambda$ is a function from $\{0,...,N-1\}$ into $\Lambda\cup\{v\}$ such that the symbol $v$ occurs at least once. For a variable word $w(v)$ over $\Lambda$ and a letter $\alpha\in\Lambda$ we denote by $w(\alpha)$ the constant word resulting by substituting each occurrence of $v$ by $\alpha$.
A combinatorial line is a set of the form $\{w(\alpha):\;\alpha\in\Lambda\}$, where $w(v)$ is a variable word over $\Lambda$.

\begin{thm}
  [Hales--Jewett]
  \label{Hales_Jewett}
  Let $k,r$ be positive integers. Then there exists a positive integer $N_0$ with the following property.
  For every positive integer $N$ with $N\meg N_0$, every alphabet $\Lambda$ with $k$ elements and every $r$-coloring of $\Lambda^N$, there exists a monochromatic combinatorial line. We denote the least such  $N_0$ by $\mathrm{HJ}(k,r)$.
\end{thm}

\subsection{The Halpern--L\"auchli Theorem}
The next result is due to
 J. D. Halpern and  H. L\"auchli \cite{HL} (see also \cite{AFK} and \cite{To} for combinatorial proofs). As we mentioned in the introduction, the main result of the present work, i.e. Theorem \ref{disjoint_Union_tree}, can be viewed as a dual form of the following.

\begin{thm}
  \label{Halpern_Lauchli}
  Let $\mathbf{T}$ be a vector tree of infinite height.
  Then for every dense 
  vector subset $\mathbf{D}$ of $\mathbf{T}$ and every subset $\mathcal{P}$ of $\otimes\mathbf{D}$, there exists a vector subset $\mathbf{D}'$ of $\mathbf{D}$ such that either
  \begin{enumerate}
    \item[(i)] $\otimes\mathbf{D}'$ is a subset of $\mathcal{P}$ and $\mathbf{D}'$ is a dense 
        vector subset of $\mathbf{T}$, or
    \item[(ii)] $\otimes\mathbf{D}'$ is a subset of $\mathcal{P}^c$ and $\mathbf{D}'$ is
        a $\ttt$-dense 
        vector subset $\mathbf{D}'$ of $\mathbf{T}$ for some $\ttt$ in $\otimes\mathbf{T}$.
  \end{enumerate}
\end{thm}


\subsection{Souslin measurability}\label{Souslin}
By $\mathcal{N}$ we denote the set of all infinite sequences in $\omega$
and by $\omega^{<\infty}$ the set of all finite sequences in $\omega$.
A Souslin scheme is a family $(X_s)_{s\in\omega^{<\infty}}$ of sets.
We say that a collection $\mathcal{C}$ of subsets of some fixed set $X$ is closed under the
Souslin operation if for every Souslin scheme $(X_s)_{s\in\omega^{<\infty}}$ consisting of
elements of $\mathcal{C}$ we have that the set
\[\bigcup_{\sigma\in\mathcal{N}}\bigcap_{n\in\omega}X_{\sigma\upharpoonright n}\]
belongs to $\mathcal{C}$.

Let $\mathcal{T}$ be a topological space. By $\mathcal{SM}(\mathcal{T})$ we denote the smallest algebra of subsets of $\mathcal{T}$
that contains all the open sets and is closed under the Souslin operation. A subset of $\mathcal{T}$ is called Souslin measurable if it belongs to $\mathcal{SM}(\mathcal{T})$.
Let us point out that $\mathcal{SM}(\mathcal{T})$ is a $\sigma$-algebra of subsets of $\mathcal{T}$ containing all the open sets. Thus, in particular, every Borel subset of $\mathcal{T}$ is Souslin measurable. Moreover, every analytic or coanalytic subset of $\mathcal{T}$ is Souslin measurable. Finally, we say that a coloring of $\mathcal{T}$ is Souslin measurable if the inverse image of every color is a Souslin measurable subset of $\mathcal{T}$.

\subsection{Abstract Ramsey Theory} \label{abstract_Ramsey_theory}
Abstract Ramsey theory was developed after several classical infinite-dimensional Ramsey-theoretic results were deduced from the corresponding one-dimensional pigeonhole principles using a similar procedure. In \cite{C} and \cite{CS2}, T. J. Carlson and S. G. Simpson were first to provide a set of axioms that underline the concept of topological Ramsey space. Here, we shall need the extension of this approach to the concept of general Ramsey space due to S. Todorcevic  (\cite{To}). We shall need to recall the approach from \cite{To}.

A \textit{Ramsey space} is a tuple of the form \[(\mathcal{R},\mathcal{S},\mik,\mik^o, r,s),\]
satisfying the following.
\begin{enumerate}
  \item[(i)] Both $\mathcal{R}$ and $\mathcal{S}$ are sets, while $\mik$ and $\mik^o$ are binary relations.
  \item[(ii)] The binary relation $\mik$ is a reflexive and transitive relation on $\mathcal{S}$, while $\mik^o$ is a subset of the cartesian product $\mathcal{R}\times\mathcal{S}$ such that for every $X,Y$ in $\mathcal{S}$ and $A$ in $\mathcal{R}$ we have that
      \[\text{if}\;A\mik^o X\;\text{and}\;X\mik Y\;\text{then}\;A\mik^o Y.\]
  \item[(iii)] There exist sets $\mathcal{AR}$ and $\mathcal{AS}$ such that
      both $r$ and $s$ are functions with
      \[r:\mathcal{R}\times\omega\to\mathcal{AR}\;\text{and}\;s:\mathcal{S}\times\omega\to\mathcal{AS}.\]

\end{enumerate}
For every non-negative integer $n$ we define $r_n:\mathcal{R}\to\mathcal{AR}$
and $s_n:\mathcal{S}\to\mathcal{AS}$ setting
\[r_n(A)=r(A,n)\;\text{and}\;s_n(X)=s(X,n)\]
for all $A$ in $\mathcal{R}$ and $X$ in $\mathcal{S}$. Moreover, for every non-negative integer $n$, we set $\mathcal{AR}_n$ to be the range of $r_n$ and $\mathcal{AS}_n$ to be the range of
$s_n$.
Without loss of generality, we may assume that $\mathcal{AR}=\bigcup_{n=0}^\infty\mathcal{AR}_n$ and $\mathcal{AS}=\bigcup_{n=0}^\infty\mathcal{AS}_n$.

We consider four axioms about a Ramsey space. The first axiom is the following.
\begin{enumerate}
  \item[\textbf{A.1}] (Sequencing) For any choice $(p,\mathcal{P})$ in $\{(r,\mathcal{R}), (s,\mathcal{S})\}$, the following are satisfied for every $P,Q$ in $\mathcal{P}$.
      \begin{enumerate}
        \item[(i)] $p_0(P)=p_0(Q)$.
        \item[(ii)] If $P\neq Q$ then $p_n(P)\neq p_n(Q)$ for some non-negative integer $n$.
        \item[(iii)] If for some pair of non-negative integers $n,m$ we have that $p_n(P)=p_m(Q)$ then $n=m$ and $p_k(P)=p_k(Q)$ for all $0\mik k\mik n$.
      \end{enumerate}
\end{enumerate}

Let us observe that if a Ramsey space satisfies axiom \textbf{A.1} then each element $A$ of $\mathcal{R}$ (resp. $X$ of $\mathcal{S}$) can be identified by the sequence $\big(r_n(A)\big)_{n=0}^\infty$ in $\mathcal{AR}$ (resp. the sequence $\big(s_n(X)\big)_{n=0}^\infty$ in $\mathcal{AS}$). Thus, in this case, the set $\mathcal{R}$ (resp. $\mathcal{S}$) can be viewed as a subset of the space $\mathcal{AR}^\omega$ (resp. $\mathcal{AS}^\omega$)  of all sequences in $\mathcal{AR}$ (resp. $\mathcal{AS}$). We endow the set $\mathcal{AR}$ with the discrete topology and the set $\mathcal{AR}^\omega$ with the corresponding product topology which we refer to as the metric topology. Thus the set $\mathcal{R}$ inherits a topology, called the metric topology on $\mathcal{R}$. Similarly,
we consider the metric topology on $\mathcal{AS}^\omega$.

Similarly, if a Ramsey space satisfies axiom \textbf{A.1} then each element $a$ of $\mathcal{AR}$ (resp. $x$ of $\mathcal{AS}$) can be identified by a finite sequence $\big(r_k(A)\big)_{k<n}$ for some $A$ in $\mathcal{R}$ (resp.  $\big(s_k(X)\big)_{k<m}$ for some $X$ in $\mathcal{S}$).
Let us observe that the numbers $n$ and $m$ are unique. We denote them by $|a|$ and $|x|$ respectively.

These observations give rise to the following ordering $\sqsubseteq$ defined on both $\mathcal{AR}$ and $\mathcal{AS}$. For $a$ and $b$ in $\mathcal{AR}$ we write $a\sqsubseteq b$
if there exist non-negative integers $m,n$ with $m\mik n$ and $A$ in $\mathcal{R}$ such that $a=r_m(A)$ and $b=r_n(A)$. Similarly, for $x$ and $y$ in $\mathcal{AS}$ we write $x\sqsubseteq y$
if there exist non-negative integers $m,n$ with $m\mik n$ and $X$ in $\mathcal{S}$ such that $x=r_m(X)$ and $Y=r_n(X)$.

The second axiom about a Ramsey space is the following.
\begin{enumerate}
  \item[\textbf{A.2}] (Finitization) There exist a binary relation $\mik^o_\mathrm{fin}$, subset of the cartesian product $\mathcal{AR}\times\mathcal{AS}$ and a reflexive and transitive binary relation $\mik_\mathrm{fin}$ on $\mathcal{AS}$ satisfying the following.
      \begin{enumerate}
        \item[(1)] For every $x$ in $\mathcal{AS}$ both the sets $\{a\in\mathcal{AR}:a\mik^o_\mathrm{fin}x\}$ and
            $\{y\in\mathcal{AS}:y\mik_\mathrm{fin}x\}$ are finite.
        \item[(2)] For every $Y,X$ in $\mathcal{S}$ we have that $Y\mik X$ if and only if for every $n$ there exists $m$ satisfying $s_n(X)\mik_\mathrm{fin}s_m(Y)$.
        \item[(3)] For every $A$ in $\mathcal{A}$ and $X$ in $\mathcal{S}$ we have that $A\mik^o X$ if and only if for every $n$ there exists $m$ satisfying $s_n(A)\mik^o_\mathrm{fin}s_m(X)$.
        \item[(4)] For every $a$ in $\mathcal{AR}$ and $x,y$ in $\mathcal{AS}$, we have that $a\mik_\mathrm{fin}^o x\mik_\mathrm{fin}y$ implies $a\mik^o_\mathrm{fin}y$.
        \item[(5)] For every $a,b$ in $\mathcal{AR}$ and $x$ in $\mathcal{AS}$, we have that $a\sqsubseteq b$ and $b\mik^o_\mathrm{fin}x$ implies the existence of some $y\sqsubseteq x$ satisfying $a\mik^o_\mathrm{fin}y$.
      \end{enumerate}
\end{enumerate}

To state the third axiom, we need some additional notation. Let $a$ in $\mathcal{AR}$, $x$ in $\mathcal{AS}$, $m$ in $\omega$ and $Y$ in $\mathcal{S}$. We define the \textit{basic sets} as follows.
\[\begin{split}
  &[a,Y]=\{A\in\mathcal{R}:\;A\mik^oY\;\text{and there exists}\;n\;\text{in}\;\omega\;\text{such that}\;r_n(A)=a\},\\
  &[x,Y]=\{X\in\mathcal{S}:\;X\mik Y\;\text{and there exists}\;n\;\text{in}\;\omega\;\text{such that}\;s_n(X)=x\}\;\text{and}\\
  &[m,Y]=[s_m(Y),Y].
\end{split}\]
Moreover, for every $a$ in $\mathcal{AR}$ and $Y$ in $\mathcal{S}$ we set
\[
\mathrm{depth}_Y(a)=\left\{ \begin{array} {l} \min\big\{k\in\omega:a\mik^o_\mathrm{fin}s_k(Y)\big\}\;\;\;\;\text{if}\;a\mik^o_\mathrm{fin}s_k(Y)\;\text{for some}\;k\in\omega\\
\infty\;\;\;\;\;\;\;\;\;\;\;\;\;\;\;\;\;\;\;\;\;\;\;\;\;\;\;\;\;\;\;\;\;\;\;\;\;\;\;\;\;\;\text{otherwise}.\end{array}  \right.
\]
The third axiom about a Ramsey space is the following.
\begin{enumerate}
  \item[\textbf{A.3}] (Amalgamation)
  \begin{enumerate}
    \item[(1)] For every $a$ in $\mathcal{AR}$ and $Y$ in $\mathcal{S}$ we have that if $d=\mathrm{depth}_Y(a)<\infty$ then every $X$ in $[d,Y]$ satisfies $[a,X]\neq\emptyset$.
    \item[(2)] For every $a$ in $\mathcal{AR}$ and $X,Y$ in $\mathcal{S}$ we have that if $X\mik Y$ and $[a,X]\neq\emptyset$ then there exists $Y'$ in $[\mathrm{depth}_Y(a),Y]$ satisfying $[a,Y']\subseteq[a,X]$.
  \end{enumerate}
\end{enumerate}

Finally, the forth axiom about a Ramsey space is the following.
\begin{enumerate}
  \item[\textbf{A.4}] (Pigeonhole) For every $a$ in $\mathcal{AR}$, every subset $\mathcal{O}$ of $\mathcal{AR}_{l+1}$, where $l=|a|$, and every $Y$ in $\mathcal{S}$ with $[a,Y]\neq\emptyset$, there exists $X$ in $[\mathrm{depth}_Y(a),Y]$ such that either $r_{l+1}[a,X]\subseteq\mathcal{O}$ or $r_{l+1}[a,X]\subseteq\mathcal{O}^c$.
\end{enumerate}
We will need the following consequence of the Abstract Ramsey Theorem.
\begin{thm}
  \label{abstract_Ramsey}
  Let $(\mathcal{R},\mathcal{S},\mik,\mik^o, r,s)$ be a Ramsey space satisfying axioms A.1--A.4.
  Assume also that $\mathcal{R}$ is a closed subset of $\mathcal{AR}^\omega$. Then for every finite Souslin measurable coloring of $R$ and $X$ in $\mathcal{S}$, there exists $X'$ in $\mathcal{S}$ with $X'\mik X$ such that the set $\{A\in\mathcal{R}:A\mik^oX'\}$ is monochromatic.
\end{thm}

\section{Notation for words on level products}\label{section_notation_words}
In this section we introduce the notation needed to state Theorem \ref{tree_Hales_Jewett}. As we mentioned in the introduction, Theorem \ref{tree_Hales_Jewett} is the analogue of the Hales--Jewett Theorem for maps defined on the level product of a vector tree. Sections \ref{section_subspaces}, \ref{section_large_sets} and \ref{section_proof_tree_Hales_Jewett} are
devoted to the proof of Theorem \ref{tree_Hales_Jewett}.

Let us fix a finite set $\Lambda$ that we view as a finite alphabet and a vector tree $\mathbf{T}=(T_1,...,T_d)$ of infinite height.
Also let $m,n$ be two non-negative integers, with $m\mik n$. We define the set of constant words $\w(m,n)$ over $\Lambda$ on $\otimes\mathbf{T}\upharpoonright[m,n)$ to be the set of
all functions from $\otimes\mathbf{T}\upharpoonright[m,n)$ into $\Lambda$, i.e.,
\[\w(\Lambda,\mathbf{T},m,n)=\Lambda^{\otimes\mathbf{T}\upharpoonright[m,n)}.\]
Let us point out that if $m=n$, then $\w(\Lambda,\mathbf{T},m,n)=\{\emptyset\}$. Moreover, we define the set of all constant words to be $\w(\Lambda,\mathbf{T})=\bigcup_{m\mik n}\w(\Lambda,\mathbf{T},m,n)$.

Let $(v_\sss)_{\sss\in\otimes\mathbf{T}}$ be a
collection of distinct symbols not occurring in $\Lambda$. The elements of this collection serve as variables.
For every vector level subset $\mathbf{D}$ of $\mathbf{T}$,
we define
$\w_v(\Lambda,\mathbf{T},\mathbf{D},m,n)$ to be the set of all functions $f$ from $\otimes\mathbf{T}\upharpoonright[m,n)$
 into $\Lambda\cup\{v_\sss:\sss\in\otimes\mathbf{D}\}$ satisfying the following.
\begin{enumerate}
  \item[(i)] The set $f^{-1}(\{u_\sss\})$ is nonempty and admits $\sss$ as a minimum in $\otimes\mathbf{T}$, for all $\sss$ in $\otimes\mathbf{D}$.
  \item[(ii)] For every $\sss$ and $\sss'$ in $\otimes\mathbf{D}$, we have $L_{\otimes\mathbf{T}}(f^{-1}(\{u_\sss\}))=L_{\otimes\mathbf{T}}(f^{-1}(\{u_{\sss'}\}))$.
\end{enumerate}
Moreover, we set
\[\begin{split}
\w_v(\Lambda,\mathbf{T},m,n)=\bigcup \big\{\w_v(\Lambda,\mathbf{T},\mathbf{D},m,n):\;&\mathbf{D}\;\text{is a vector level subset of}\;\mathbf{T}\\
&\text{with}\;L_\mathbf{T}(\mathbf{D})\subset[m,n)\big\}.
\end{split}\]
For every $f$ in $\w_v(\Lambda,\mathbf{T},m,n)$, we set $\ws(f)=\mathbf{D}$, where $\mathbf{D}$ is the unique vector level subset of $\mathbf{T}$ such that $f\in\w_v(\Lambda,\mathbf{T},\mathbf{D},m,n)$.
Finally, we define $\w_v(\Lambda,\mathbf{T})=\bigcup_{m<n}\w_v(\Lambda,\mathbf{T},m,n)$.
The elements of $\w_v(\Lambda,\mathbf{T})$ are viewed as variable words over the alphabet $\Lambda$.

Observe that for every nonempty element $f$ in $\w(\Lambda,\mathbf{T})\cup\w_v(\Lambda,\mathbf{T})$, there exists unique non-negative
integers $m,n$ with $m<n$ such that the function $f$ belongs to $\w(\Lambda,\mathbf{T},m,n)\cup\w_v(\Lambda,\mathbf{T},m,n)$.
We set $\bt(f)=m$ and $\tp(f)=n$.

For variable words, we consider substitutions. In particular, for every $f$ in $\w_v(\Lambda,\mathbf{T})$ and every family $\mathbf{a}=(a_\sss)_{\sss\in\otimes\ws(f)}$ of elements in $\Lambda$, we define $f(\mathbf{a})$ to be the unique element of $\w(\Lambda,\mathbf{T})$ resulting by substituting each occurrence
of $v_\sss$ by $a_\sss$, for every $\sss$ in $\otimes\ws(f)$.
Moreover, for every word, variable or not, we consider its span.
In particular, for every $f$ in $\w_v(\Lambda,\mathbf{T})$
we set
\[[f]_\Lambda=\{f(\mathbf{a}):\;\mathbf{a}=(a_\sss)_{\sss\in\otimes\ws(f)}\;\text{is a collection of elements from}\;\Lambda\},\]
while for every $f$ in $\w(\Lambda,\mathbf{T})$, we set $[f]_\Lambda=\{f\}$.

Let $f_1$ and $f_2$ in $\w(\Lambda,\mathbf{T})\cup\w_v(\Lambda,\mathbf{T})$. We say that the pair $(f_1,f_2)$ is compatible
if either at least one of $f_1$ and $f_2$ is the empty function or $\tp(f_1)=\bt(f_2)$. A finite sequence $(f_i)_{i=0}^{q-1}$ in $\w(\Lambda,\mathbf{T})\cup\w_v(\Lambda,\mathbf{T})$ is called compatible, if for every $0\mik i<q-1$ the pair $(f_i,f_{i+1})$ is compatible.
Similarly, an infinite sequence $(f_i)_{i=0}^\infty$ in $\w(\Lambda,\mathbf{T})\cup\w_v(\Lambda,\mathbf{T})$, is called compatible, if for every non-negative integer $i$ we have that the pair $(f_i,f_{i+1})$ is compatible.

We extend the notion of span for compatible sequences. Let $m_1,m_2,m_3$ be three non-negative integers, with $m_1\mik m_2\mik m_3$. Also let $A$ be a subset of the set $\w(\Lambda,\mathbf{T},m_1,m_2)\cup\w_v(\Lambda,\mathbf{T},m_1,m_2)$ and $B$ a subset of the set $\w(\Lambda,\mathbf{T},m_2,m_3)\cup\w_v(\Lambda,\mathbf{T},m_2,m_3)$. We
set
\[A^\con B=\{f_1\cup f_2: f_1\in A\;\text{and}\;f_2\in B\}.\]
For a compatible finite sequence $(f_i)_{i=0}^{q-1}$ in $\w(\Lambda,\mathbf{T})\cup\w_v(\Lambda,\mathbf{T})$, we define its span as
\[[(f_i)_{i=0}^{q-1}]_\Lambda=[f_0]_\Lambda^\con...^\con[f_{q-1}]_\Lambda,\]
while for a compatible infinite sequence $(f_i)_{i=0}^\infty$, we define its span as
\[[(f_i)_{i=0}^\infty]_\Lambda=\bigcup_{q=1}^\infty[(f_i)_{i=0}^{q-1}]_\Lambda.\]

A compatible infinite sequence $X=(f_i)_{i=0}^\infty$ is called a subspace (of $\w(\Lambda,\mathbf{T})$), if the following are satisfied.
\begin{enumerate}
  \item[(i)] For every non-negative integer $i$, we have that $f_i$ belongs to $\w_v(\Lambda,\mathbf{T})$.
  \item[(ii)] For every non-negative integer $i$, we have that the vector level subset $\ws(f_i)$ dominates $\mathbf{T}(i)$.
  \item[(iii)] $\bt(f_0)=0$.
\end{enumerate}
For $Y$ and $X$ two subspaces, we say that $Y$ is a further subspace of $X$, we write $Y\mik X$, if $[Y]_\Lambda$ is a subset of $[X]_\Lambda$.
We prove the following.

\begin{thm}\label{tree_Hales_Jewett}
  Let $\Lambda$ be a finite alphabet and $\mathbf{T}$ a vector tree of infinite height. Then for every finite coloring of the set of the constant words $\w(\Lambda,\mathbf{T})$ over $\Lambda$ and every subspace $X$ of $\w(\Lambda,\mathbf{T})$ there exists a subspace $X'$ of $\w(\Lambda,\mathbf{T})$ with $X'\mik X$ such that the set $[X']_\Lambda$ is monochromatic.
\end{thm}

\section{Subspaces}\label{section_subspaces}
For the proof of Theorem \ref{tree_Hales_Jewett} we shall need some additional notation concerning the subspaces.
Let us fix throughout this section a finite alphabet $\Lambda$ and a vector tree $\mathbf{T}$ of infinite height.
For notational simplicity we will write $\w$ (resp. $\w_v$) instead of $\w(\Lambda,\mathbf{T})$ (resp. $\w_v(\Lambda,\mathbf{T})$).
For a non-negative integer $\ell$, we define \[\w(\ell)=\{f\in\w:\;\text{if}\;f\neq\emptyset\;\text{then}\;\bt(f)=\ell\}\;\text{and}\;\w_v(\ell)=\{f\in\w_v:\;\bt(f)=\ell\}.\]

For the proof of Theorem \ref{tree_Hales_Jewett} we need to enlarge the class of the subspaces that we are looking at. Let $k,\ell$ be two non negative integers.
A compatible infinite sequence $X=(f_i)_{i=0}^\infty$ is called a $(k,\ell)$-subspace, if the following are satisfied.
\begin{enumerate}
  \item[(i)] For every non-negative integer $i$, we have that $f_i$ belongs to $\w_v$.
  \item[(ii)] For every non-negative integer $i$, we have that the vector level subset $\ws(f_i)$ dominates $\mathbf{T}(k+i)$.
  \item[(iii)] $\bt(f_0)=\ell$.
\end{enumerate}
Let as observe that if $X$ is a $(k,\ell)$-subspace, then $X$ is also a $(k',\ell)$-subspace for every non-negative integer $k'$ with $k'\mik k$.
Moreover, $X$ is a subspace if and only if $X$ is a $(0,k')$-subspace for some non-negative integer $k'$.
For a non-negative integer $\ell$, we say that $X$ is an $\ell$-subspace, if $X$ is a $(k,\ell)$-subspace
for some non-negative integer $k$. Thus, $X$ is a subspace if and only if $X$ is a $0$-subspace.
Moreover, a compatible finite sequence $\xx$ in $\w_v$ is called
a finite $(k,\ell)$-subspace, if there exists a $(k,\ell)$-subspace having $\xx$ as an initial segment. Similarly, a compatible finite sequence $\xx$ in $\w_v$ is called
a finite $\ell$-subspace, if there exists an $\ell$-subspace having $\xx$ as an initial segment.

Let us fix some non-negative integer $\ell$. We say that an $\ell$-subspace $Y$ is a further subspace of $X$, where $X$ is an $\ell$-subspace too, if $[Y]_\Lambda$ is a subset of $[X]_\Lambda$.
In this case we write $Y\mik X$.
Similarly, if $\yy$ is a finite $\ell$-subspace and
$X$ an $\ell$-subspace (resp. $\xx$ another finite $\ell$-subspace) we write $\yy\mik X$
(resp. $\yy\mik\xx$) if $[\yy]_\Lambda$ is subset of $[X]_\Lambda$ (resp. $[\xx]_\Lambda$).

Let $\ell$ be a non-negative integer. Also let $\xx$ and $\yy$ be two finite $\ell$-subspaces.
We say that $\yy$ extends $\xx$, we write $\xx\sqsubseteq\yy$, if $\xx$ is an initial segment of $\yy$.
Similarly, if $X$ is an $\ell$-subspace, we write $\xx\sqsubseteq X$ if $\xx$ is an initial segment of $X$. Moreover, we will say that $\xx$ is a by one extension of $\yy$, if $\xx$ extends $\yy$ and its length equals the length of $\yy$ plus one.

Let $\ell$ be a non-negative integer and $X=(f_i)_{i=0}^\infty$ be an $\ell$-subspace. Also let $\xx$
be a finite $\ell$-subspace such that $\xx\mik X$.
Then observe that there exists a unique positive integer $q$, such that $\xx\mik(f_i)_{i=0}^{q-1}$. We define $X/\xx=(f_i)_{i=q}^\infty$. Observe that $X/\xx$ is an $\ell'$-subspace, where $\ell'=\bt(f_q)$.

\section{Large sets}\label{section_large_sets}
The main notion that helps us to carry out the proof of Theorem \ref{tree_Hales_Jewett} is the one  of largeness. In this section, we include the definition of the large set and the results related to it. These results form the main part of the proof of Theorem  \ref{tree_Hales_Jewett}.

Let us fix throughout this section a finite alphabet $\Lambda$ and a vector tree $\mathbf{T}$ of infinite height.
\begin{defn}
  \label{defn_large}
  Let $\ell$ be a non-negative integer. Also let $E$ a subset of $\w(\ell)$ and $X$ an $\ell$-subspace.
  We say that $E$ is large in $X$ if for every further subspace $Y$ of $X$ we have that $E\cap[Y]_\Lambda\neq\emptyset$.
\end{defn}

We have the following easy to observe facts concerning the notion of largeness. The first claims that the notion of largeness is hereditary.
\begin{fact}
  \label{hereditary_fact}
  Let $\Lambda$ be a finite alphabet and $\mathbf{T}$ a vector tree of infinite height. Also let $\ell$ be a non-negative integer, $X$ an $\ell$-subspace and $E$ a subset of $\w(\ell)$ such that $E$ is large in $X$.
  Then for every $Y\mik X$, we have that $E$ is large in $Y$.
\end{fact}

The second fact establishes a mild Ramsey property for the notion of largeness.

\begin{fact}
  \label{mind_Ramsey_fact}
  Let $\Lambda$ be a finite alphabet and $\mathbf{T}$ a vector tree of infinite height.
  Also let $\ell$ be a non-negative integer, $X$ an $\ell$-subspace and $E$ a subset of $\w(\ell)$ such that $E$ is large in $X$.
  Finally, let $r$ be a positive integer and $E=\bigcup_{i=1}^{r}E_i$ a partition of $E$. Then there exist $i_0\in\{1,..,r\}$ and $Y\mik X$ such that $E_{i_0}$ is large in $Y$.
\end{fact}
We shall need some additional notation. Let $\ell$ be a non-negative integer and $E$ be a subset of $\w(\ell)$. Then for every $f$ in $\w(\ell)\cup\w_v(\ell)$ we set
\[E_f=\big\{g\in\w\big(\tp(f)\big):\;f'\cup g\in E\;\text{for all}\;f'\in[f]_\Lambda\big\}\]
if $f\neq\emptyset$ and $E_f=E$ otherwise.
Similarly, for a finite $\ell$-subspace $\xx$ we set
\[E_\xx=\bigcap_{f\in[\xx]_\Lambda}E_f.\]


We have the following lemma.
\begin{lem}
  \label{pushing_large_lem}
  Let $\Lambda$ be a finite alphabet and $\mathbf{T}$ a vector tree of infinite height.
  Also let $k,\ell$ be non-negative integers, $X$ an $\ell$-subspace and $E$ a subset of $\w(\ell)$ such that $E$ is large in $X$. Then there exist $g$ in $\w_v(\ell)$ and a  $\tp(g)$-subspace $Y$ satisfying the following.
  \begin{enumerate}
    \item[(a)] $[g]_\Lambda$ is a subset of $[X]_\Lambda$.
    \item[(b)] $\ws(g)$ dominates $\mathbf{T}(k)$.
    \item[(c)] $Y\mik X/(g)$.
    \item[(d)] $E_g$ is large in $Y$.
  \end{enumerate}
\end{lem}
\begin{proof}
  Let $\mathcal{X}$ be the set of all finite $\ell$-subspaces $\xx$ such that either $\xx$ is the
  empty sequence or $\xx\mik X$ and if $\xx=(f_i)_{i=0}^{q-1}$ then the cardinality of the set $\otimes\ws(f_i)$ is equal  to $\otimes\mathbf{T}(i)$ for all $0\mik i<q$.
  We have the following claim.

  \noindent\textbf{Claim 1:}
  There exists a finite $\ell$-subspace $\xx$ in $\mathcal{X}$ such that for every $\xx'$ in $\mathcal{X}$ extension by one of
  $\xx$ we have that $[\xx']_\Lambda\cap E\neq\emptyset$.
  \begin{proof}
    [Proof of Claim 1] Assume on the contrary that for every $\xx$ in $\mathcal{X}$
    there exists an extension by one
  $\xx'$ of $\xx$ in $\mathcal{X}$ such that $[\xx']_\Lambda\cap E=\emptyset$. Then we construct an infinite sequence
  $(\xx_n)_{n=0}^{\infty}$ in $\mathcal{X}$ satisfying for every non-negative integer $n$ the following.
  \begin{enumerate}
    \item[(i)] The length of $\xx_n$ is $n+1$.
    \item[(ii)] $\xx_n\sqsubseteq\xx_{n+1}$.
    \item[(iii)] $[x_n]_\Lambda\cap E=\emptyset$.
  \end{enumerate}
  Indeed let $\xx'_0$ to be the empty sequence. Then by our assumption there exists an extension by one $\xx_0$ of $\xx'_0$ in $\mathcal{X}$ such that $[\xx_0]_\Lambda\cap E=\emptyset$. By this choice (i) and (iii) above are satisfied, while (ii) is meaningless. Assume that for some non-negative integer $n$ the elements $\xx_0,...,\xx_n$ have been chosen properly. By our assumption there exists an extension by one $\xx_{n+1}$ of $\xx_n$ in $\mathcal{X}$ such that $[\xx_{n+1}]_\Lambda\cap E=\emptyset$. Clearly, $\xx_{n+1}$ is as desired.

  Let us observe that by property (ii)  there exists unique $\ell$-subspace $Y$ such that $\xx_n$ is an initial segment of $Y$ for all $n$. Moreover, by property (i), we have that $[Y]_\Lambda=\bigcup_{n=0}^\infty[\xx_n]_\Lambda$. Thus, invoking property (iii), we get that $[Y]_\Lambda\cap E=\emptyset$. Finally, for every non-negative integer $n$, since $\xx_n$ belongs to $\mathcal{X}$, we have that $\xx_n\mik X$. Thus $Y\mik X$. Let us observe that
  the existence of such a $Y$ contradicts that $E$ is large in $X$.
  \end{proof}

  Let $q$ be the length of $\xx$provided by Claim 1. We set $r=|\Lambda|^{|\otimes\mathbf{T}\upharpoonright q+1|}$. Observe that for every extension by one $\xx'$ of
  $\xx$ in $\mathcal{X}$, we have that $[\xx']_\Lambda$ is of cardinality $r$. Moreover, we set $N=\mathrm{HJ}(|\Lambda|^{|\otimes\mathbf{T}(k)|},r)$. We pick a compatible sequence $\mathbf{g}=(g_i)_{i=0}^{N-1}$ in $\w_v$ satisfying the following.
  \begin{enumerate}
    \item[(1)] $\mathbf{g}\mik X/\xx$.
    \item[(2)] $\ws(g_0)$ dominates $\mathbf{T}(k)$ and $\ws(g_{i+1})$ dominates $\ws(g_i)$ for all $0\mik i<N-1$.
    \item[(3)] $|\otimes\ws(g_i)|=|\otimes\mathbf{T}(k)|$ for all $0\mik i<N$.
  \end{enumerate}
  We set $\mathcal{L}$ to be the set of all $h$ in $\w_v$ such that $[h]_\Lambda$ is subset of $[\mathbf{g}]_\Lambda$ and $\ws(h)$ dominates $\mathbf{T}(k)$. Let us observe that for every $h$
  in $\mathcal{L}$ we have that $|\otimes\ws(h)|=|\otimes\mathbf{T}(k)|$ and the pair $(f,h)$ is compatible for every $f$ in $[\xx]_\Lambda$. Finally, by (1) above, we have that there exists a
  $\tp(g_{N-1})$-subspace $X'$ such that $[\mathbf{g}^\con X']_\Lambda$ is a subset of $[X/\xx]_\Lambda$.
  We have the following claim.

  \noindent \textbf{Claim 2:} The set $\bigcup_{f\in[\xx]_\Lambda,h\in\mathcal{L}}E_{f\cup h}$ is large in $X'$.
  \begin{proof}
    [Proof of Claim 2]
    Let $X''\mik X'$ be arbitrary. It suffices to find $f$ in $[\xx]_\Lambda$, $h$ in $\mathcal{L}$ and $f'$ in $[X'']_\Lambda$ such that $f'\in E_{f\cup h}$.
    Pick $w$ in $\w_v$ such that
    \begin{enumerate}
      \item[(a)] $[w]_\Lambda$ is a subset of $[X'']_\Lambda$.
      \item[(b)] The vector level subset $\ws(w)$ dominates $\mathbf{T}(q)$.
      \item[(c)] The cardinality of $\otimes\ws(w)$ equals the cardinality of $\otimes\mathbf{T}(q)$.
    \end{enumerate}
    Notice that for every $g$ in $[\mathbf{g}]_\Lambda$, we have that $\xx^\con (g\cup w)$ is an extension by one of $\xx$ in $\mathcal{X}$ and therefore, by Claim 1, there exists a pair $(f_g,f'_g)$ in the cartesian product $[\xx]_\Lambda\times[w]_\Lambda$ such that $f_g\cup g\cup f'_g$ belongs to $E$. Moreover, notice that $[\xx]_\Lambda\times[w]_\Lambda$ is of cardinality $r$.
    We set
\[\Lambda'=\big\{(a_\ttt)_{\ttt\in\otimes\mathbf{T}(k)}:\;a_\ttt\in\Lambda\;\text{for all}\;\ttt\in\otimes{T}(k)\big\}\]
and $\mathcal{A}$ to be the set of all sequence $(\mathbf{a}_i)_{i=0}^{N-1}$ of length $N$ with elements from $\Lambda'$.  For every $i$ in $\{0,...,N-1\}$ and $\sss$ in $\otimes\ws(g_i)$ we set
$\ttt_\sss$ to be the unique element of $\otimes\mathbf{T}(k)$ such that
$\ttt\mik_{\otimes\mathbf{T}}\ttt_\sss$.
We define a map $Q:\mathcal{A}\to[\mathbf{g}]_\Lambda$ setting
\[Q\Big(\big((a^i_\ttt)_{\ttt\in\otimes\mathbf{T}(k)}\big)_{i=0}^{N-1}\Big)=\bigcup_{i=0}^{N-1}g_i\big((a^i_{\ttt_\sss})_{\sss\in\otimes\ws(g_i)}\big),\]
for all $\big((a^i_\ttt)_{\ttt\in\otimes\mathbf{T}(k)}\big)_{i=0}^{N-1}$ in $\mathcal{A}$.
Observe that for every combinatorial line $\mathbb{L}$ of $\mathcal{A}$ we have that there exists unique element $h$ of $\mathcal{L}$ such that $Q[\mathbb{L}]=[h]_\Lambda$. 

We define a coloring $c:\mathcal{A}\to[\xx]_\Lambda\times[w]_\Lambda$ by the rule
\[c(\overline{\mathbf{a}})=(f_{Q(\overline{\mathbf{a}})},f'_{Q(\overline{\mathbf{a}})})\]
for all $\overline{\mathbf{a}}$ in $\mathcal{A}$.

By the choice of $r,N$ and the Hales--Jewett theorem, that is, Theorem \ref{Hales_Jewett}, we have that there exists a $c$-monochromatic combinatorial line $\mathbb{L}$ of $\mathcal{A}$. Let $h$ be the unique element of $\mathcal{L}$
satisfying $Q[\mathbb{L}]=[h]_\Lambda$. Let $f$ in $[\xx]_\Lambda$ and $f'$ in $[w]_\Lambda$ such that for every $\overline{\mathbf{a}}$ in $\mathbb{L}$, we have that $c(\overline{\mathbf{a}})=(f,f')$. Observe that for every $\overline{\mathbf{a}}$ in $\mathbb{L}$ we have that $f\cup Q(\overline{\mathbf{a}})\cup f'$ belongs to $E$. By the choice of $h$, we have for every $g\in[h]_\Lambda$ that $f\cup g\cup f'$ belongs to $E$. That is, $f'$ belongs to $E_{f\cup h}$. Finally, by (a), we have that $f'$ is element of $[X'']_\Lambda$. The proof of Claim 2 is complete.
  \end{proof}
By Claim 2 and Fact \ref{mind_Ramsey_fact}, there exist $f$ in $[\xx]_\Lambda$, $h$ in $\mathcal{L}$ and $Y\mik X'$ such that $E_{f\cup h}$ is large in $Y$. Setting $g=f\cup h$, it follows easily that $g$ and $Y$ are as desired.
\end{proof}

An iteration of the above theorem yields the following.
\begin{cor}
  \label{full_large_cor}
  Let $\Lambda$ be a finite alphabet and $\mathbf{T}$ a vector tree of infinite height.
  Also let $\ell$ be a non-negative integer, $X$  an $\ell$-subspace and $E$ a subset of $\w(\ell)$ such that $E$ is large in $X$.
  Then there exists an $\ell$-subspace $Y$ with $Y\mik X$ such that for every  finite $\ell$-subspace $\xx$ with $\xx\mik Y$ we have that $E_\xx$ is large in $Y/\xx$.
\end{cor}
\begin{proof}
  By an iterated use of Lemma \ref{pushing_large_lem}, we inductively construct a sequence $(g_n)_{n=0}^\infty$ in $\w_v$ and a sequence $(Y_n)_{n=0}^{\infty}$ with $Y_0=X$ satisfying for every non-negative integer $n$
  the following.
  \begin{enumerate}
    \item[(a)] $[g_n]_\Lambda$ is a subset of $[Y_n]_\Lambda$.
    \item[(b)] $Y_{n+1}\mik Y_n/(g_n)$.
    \item[(c)] $\ws(g_n)$ dominates $\mathbf{T}(n)$.
    \item[(e)] $E_{(g_i)_{i=0}^n}$ is large in $Y_{n+1}$.
  \end{enumerate}
  We set $Y=(g_i)_{i=0}^\infty$. Then $Y$ is as desired. Indeed, first, observe that
  $Y$ is an $\ell$-subspace with $Y\mik X$.
  Let
  $\yy_n=(g_i)_{i=0}^n$ for all $n\meg0$. Notice that for every non-negative integer $n$ we
  have that
  \begin{equation}
    \label{eq01}
    Y/\yy_n\mik Y_{n+1}.
  \end{equation}
  Finally, let $\xx$ be a finite $\ell$-subspace with $\xx\mik Y$.
  Then there exists some $n$ such that $\xx\mik \yy_n$. Since $[\xx]_\Lambda\subseteq[\yy_n]_\Lambda$, we have that $E_{\yy_n}\subseteq E_\xx$.
  Moreover, we have that $Y/\xx=Y/\yy_n$. Invoking (e)
  and \eqref{eq01}, we have that $E_\xx$ is large in $Y/\xx$.
\end{proof}

Finally, we will need the following lemma concerning the large sets.

\begin{lem}
    \label{getting_inside_large}
  Let $\Lambda$ be a finite alphabet and $\mathbf{T}$ a vector tree of infinite height.
  Also let $k,\ell$ be non-negative integers, $X$ an $\ell$-subspace and $E$ a subset of $\w(\ell)$ such that $E$ is large in $X$. Then there exists $f$ in $\w_v$ satisfying the following.
  \begin{enumerate}
    \item[(i)] The set $[f]_\Lambda$ is a subset of $E\cap[X]_\Lambda$.
    \item[(ii)] The vector level subset $\ws(f)$ dominates $\mathbf{T}(k)$.
  \end{enumerate}
\end{lem}
\begin{proof}
  Let as assume on the contrary. Then for every integer $k'\meg k$, since $\mathbf{T}(k')$ dominates $\mathbf{T}(k)$, there is no $f$ in $\w_v$ satisfying condition (i) above such that the vector level subset $\ws(f)$ dominates $\mathbf{T}(k')$.

  We construct inductively a compatible sequence $(f_n)_{n=0}^\infty$ in
  $\w_v$ satisfying the following.
  \begin{enumerate}
    \item[(C1)] The set $[f_0]_\Lambda$ is a subset of $[X]_\Lambda\setminus E$.
    \item[(C2)] For every positive integer $n$ and every $g$ in $[(f_i)_{i=0}^{n-1} ]_\Lambda$, we have that the set $[g\cup f_n]_\Lambda$ is a subset of $[X]_\Lambda\setminus E$.
    \item[(C3)] The vector level subset $\ws(f_n)$ dominates $\mathbf{T}(k+n)$, for all $n$ in $\omega$.
  \end{enumerate}
  First, let us observe that the above construction leads to a contradiction and therefore completes the proof of the lemma. Indeed, set $Y=(f_n)_{n=0}^\infty$. Since $X$ in an $\ell$-subspace, by (C1), we have, in particular, that $f_0$ belongs to $\w(\ell)$. Invoking (C3), we get that $Y$ is an $\ell$-subspace. Observe that for every positive integer $n$ we have that
  \[
        \big[(f_i)_{i=0}^{n}\big]_\Lambda=\bigcup_{g\in \big[(f_i)_{i=0}^{n-1}\big]_\Lambda}\{g\}^\con[f_n]_\Lambda.
  \]
  Thus, invoking (C1) and (C2), we have that $Y\mik X$ and $[Y]_\Lambda\cap E=\emptyset$,
  which contradicts the largeness of $E$ in $X$.

  We will describe the inductive step of the construction. The initial step is similar to the general one.
  Let as assume that for some positive integer $n$ the words $f_0,...,f_{n-1}$ have been chosen properly.

  Let $\xx=(f_i)_{i=0}^{n-1}$ and let $r$ be the cardinality of the set $[\xx]_\Lambda$. We set $N=\mathrm{HJ}(|\Lambda|^{|\otimes \mathbf{T}(k+n)|},2^r)$
   and
\[\Lambda'=\big\{(a_\ttt)_{\ttt\in\otimes\mathbf{T}(k+n)}:\;a_\ttt\in\Lambda\;\text{for all}\;\ttt\in\otimes\mathbf{T}(k+n)\big\}.\]
We define $\mathcal{A}$ to be the set of all sequence $\overline{\mathbf{a}}=(\mathbf{a}_i)_{i=0}^{N-1}$ of length $N$ with elements from $\Lambda'$. Moreover, we set $\mathcal{C}$ to be the set of all maps from $[\xx]_\Lambda$ into $\{0,1\}$. Clearly $\mathcal{C}$ is of cardinality $2^r$.

We pick a compatible sequence $\mathbf{h}=(h_i)_{i=0}^{N-1}$ in $\w_v$ satisfying the following.
\begin{enumerate}
  \item[(a)] The pair $(f_{n-1}, h_0)$ is compatible.
  \item[(b)] The set $[\mathbf{h}]_\Lambda$ is a subset of $[X/\xx]_\Lambda$.
  \item[(c)] The vector level subset $\ws(h_0)$ dominates $\mathbf{T}(k+n)$.
  \item[(d)] The vector level subset $\ws(h_{i+1})$ dominates $\ws(h_{i})$ for all $0\mik i<M-1$.
  \item[(e)] $|\otimes\ws(h_i)|=|\otimes\mathbf{T}(k+n)|$ for all $0\mik i<M$.
\end{enumerate}
For every $i$ in $\{0,...,N-1\}$ and $\sss$ in $\otimes\ws(f_i)$ we set
$\ttt_\sss$ to be the unique element of $\otimes\mathbf{T}(k+n)$ such that
$\ttt_\sss\mik_{\otimes\mathbf{T}}\sss$.
We define a map $Q:\mathcal{A}\to\big[\mathbf{h}\big]_\Lambda$ setting
\[Q\Big(\big((a^i_\ttt)_{\ttt\in\otimes\mathbf{T}(k+n)}\big)_{i=0}^{N-1}\Big)=\bigcup_{i=0}^{N-1}h_i\big((a^i_{\ttt_\sss})_{\sss\in\otimes\ws(g_i)}\big),\]
for all $\big((a^i_\ttt)_{\ttt\in\otimes\mathbf{T}(k)}\big)_{i=0}^{N-1}$ in $\mathcal{A}$.

We define a coloring $c:\mathcal{A}\to\mathcal{C}$ as follows. For every $\overline{\mathbf{a}}$ in $\mathcal{A}$ and $g$ in $[\xx]_\Lambda$, we set $c(\overline{\mathbf{a}})(g)=0$ if $g\cup Q(\overline{\mathbf{a}})\not\in E$ and $c(\overline{\mathbf{a}})(g)=1$ otherwise.
By the the Hales--Jewett theorem, that is, Theorem \ref{Hales_Jewett}, we have that there exists a $c$-monochromatic combinatorial line $\mathbb{L}$ of $\mathcal{A}$.
Let $f_n$ be the unique element of $\w_v$ such that $[f_n]_\Lambda=\{Q(\overline{\mathbf{a}}):\;\overline{\mathbf{a}}\in\mathbb{L}\}$.
By (a), we have that $(f_i)_{i=0}^n$ is compatible.
By (b), we get that
\begin{equation}
  \label{eq02}
  \big[(f_i)_{i=0}^n\big]_\Lambda\subseteq[X]_\Lambda.
\end{equation}
By (c) and (d), we have that (C3) is satisfied.
By \eqref{eq02}, the fact that $\mathbb{L}$ is $c$-monochromatic and the definition of the coloring $c$, we have for every $g$ in $[\xx]_\Lambda$ that either
\begin{enumerate}
  \item[(I)] the set $[g\cup f_n]_\Lambda$ is subset of $[X]_\Lambda\setminus E$, or
  \item[(II)] the set $[g\cup f_n]_\Lambda$ is subset of $E\cap[X]_\Lambda$.
\end{enumerate}
Observe that the second alternative contradicts our initial assumption. Thus (C2) is satisfied. The proof of the inductive step of the construction is complete as well as the proof of the lemma.
\end{proof}

\section{Proof of Theorem \ref{tree_Hales_Jewett}}\label{section_proof_tree_Hales_Jewett}
 Let $\Lambda$ be a finite alphabet and $\mathbf{T}$
a vector tree of infinite height. Also let $\ell$ be a non-negative integer. We set
$\w(\Lambda,\mathbf{T},\ell)=\{f\in\w(\Lambda,\mathbf{T}):\;\text{if}\;f\neq\emptyset\;\text{then}\;\bt(f)=\ell\}$.
Actually, we will prove the following equivalent form of Theorem \ref{tree_Hales_Jewett}.

\begin{thm}\label{tree_Hales_Jewett_2}
  Let $\Lambda$ be a finite alphabet, $\mathbf{T}$ a vector tree of infinite height and $\ell$ a non-negative integer. Then for every finite coloring of the set $\w(\Lambda,\mathbf{T},\ell)$ and every $\ell$-subspace $X$ there exists an $\ell$-subspace $X'$ with $X'\mik X$ such that the set $[X']_\Lambda$ is monochromatic.
\end{thm}
\begin{proof}
  Let us fix a finite coloring of $\w(\Lambda,\mathbf{T},\ell)$. Then by Fact \ref{mind_Ramsey_fact} there exists an $\ell$-subspace $Y'$ with $Y'\mik X$ such that one of the colors forms a subset $E$ of $\w(\Lambda,\mathbf{T},\ell)$ which is large in $Y'$. Applying Corollary \ref{full_large_cor}, we obtain an $\ell$-subspace $Y$ with $Y\mik Y'$ such that for every finite $\ell$-subspace $\xx$ with $\xx\mik Y$ we have that $E_\xx$ is large in $Y/\xx$.

  Using Lemma \ref{getting_inside_large}, we inductively construct a compatible sequence
  $(f_n)_{n=0}^\infty$ in $\w_v(\Lambda,\mathbf{T})$,
a sequence $(Y_n)_{n=0}^\infty$ with $Y_0=Y$
and a sequence $(E_n)_{n=0}^\infty$ with $E_0=E$ satisfying the following for every non-negative integer $n$.
\begin{enumerate}
    \item[(i)] The set $[f_n]_\Lambda$ is a subset of $E_n\cap[Y_n]_\Lambda$.
    \item[(ii)] The vector level subset $\ws(f_n)$ dominates $\mathbf{T}(n)$.
    \item[(iii)] $Y_{n+1}=Y_n/(f_n)$ and $E_{n+1}=(E_n)_{f_n}$.
  \end{enumerate}
  Setting $X'=(f_n)_{n=0}^\infty$, it is easy to see that $X'$ is as desired.
\end{proof}

Clearly, Theorem \ref{tree_Hales_Jewett_2} has Theorem \ref{tree_Hales_Jewett} as an immediate consequence. However, these two statements are equivalent.

\section{Infinite dimensional version of Theorem \ref{tree_Hales_Jewett}}\label{section_infinite_Hales_Jewett}
In this section we obtain the infinite dimensional version of Theorem \ref{tree_Hales_Jewett}, that is, Theorem \ref{tree_Hales_Jewett_infinite} below, by an application of the abstract Ramsey theory (see Subsection \ref{abstract_Ramsey_theory}). To state Theorem \ref{tree_Hales_Jewett_infinite} we need some pieces of notation.

Let $\Lambda$ be a finite alphabet and $\mathbf{T}$ a vector tree of infinite height.
We define $\w^{[\infty]}(\Lambda,\mathbf{T})$ to be the set of all compatible sequences $(f_n)_{n=0}^\infty$ in $\w(\Lambda,\mathbf{T})\setminus\{\emptyset\}$ with $\bt(f_0)=0$.
The set $\w^{[\infty]}(\Lambda,\mathbf{T})$ forms a subset of the set $\big(\w(\Lambda,\mathbf{T})\big)^\omega$ of all sequences in $\w(\Lambda,\mathbf{T})$.
We endow the set $\w(\Lambda,\mathbf{T})$ with the discrete topology and $\big(\w(\Lambda,\mathbf{T})\big)^\omega$ with the corresponding product topology.
We consider $\w^{[\infty]}(\Lambda,\mathbf{T})$ with the induced topology.
Moreover, for every subspace $X$ we define $[X]^{[\infty]}_\Lambda$ to be the set of all
elements $(f_n)_{n=0}^\infty$ of $\w^{[\infty]}(\Lambda,\mathbf{T})$ such that for every positive integer $n$ we have that $f_0\cup...\cup f_{n-1}$ belongs to $[X]_\Lambda$.
We have the following infinite dimensional version of Theorem \ref{tree_Hales_Jewett}.
\begin{thm}
  \label{tree_Hales_Jewett_infinite}
  Let $\Lambda$ be a finite alphabet and $\mathbf{T}$ a vector tree of infinite height. Then for every finite Souslin measurable coloring of the set $\w^{[\infty]}(\Lambda,\mathbf{T})$ and every subspace $X$ there exists a subspace $X'$ with $X'\mik X$ such that the set $[X']^{[\infty]}_\Lambda$ is monochromatic.
\end{thm}
\begin{proof}
  We will built an appropriate Ramsey space such that an application of Theorem \ref{abstract_Ramsey} will imply the result.
  We set $\mathcal{S}$ to be the set of all
  subspaces of $\w(\Lambda,\mathbf{T})$ and $\mathcal{R}=\w^{[\infty]}(\Lambda,\mathbf{T})$.
  For $X,Y$ in $\mathcal{S}$ we set $X\mik Y$, if $[X]_\Lambda\subseteq[Y]_\Lambda$, as usual.
  For $X$ in $\mathcal{S}$ and $A$ in $\mathcal{R}$, we set $A\mik^o X$ if
  $A\in[X]_\Lambda^{[\infty]}$. We set $\mathcal{AR}$ to be the set of all finite sequences $a$, possibly empty,
  in $\w(\Lambda,\mathbf{T})$ such that $a$ is an initial segment of some $A$ in $\mathcal{R}$.
  Similarly, we set $\mathcal{AS}$ to be the set of all finite sequences $x$, possibly empty, in
  $\w_v(\Lambda,\mathbf{T})$ such that $x$ is an initial segment of some $X$ in $\mathcal{S}$.
  We define $r:\mathcal{R}\times\omega\to\mathcal{AR}$,
  setting $r(A,n)$ to be the initial segment of $A$ of length $n$, for all $A$ in $\mathcal{R}$
  and $n$ in $\omega$. Similarly, we define $s:\mathcal{S}\times\omega\to\mathcal{AS}$,
  setting $s(X,n)$ to be the initial segment of $X$ of length $n$, for all $X$ in $\mathcal{S}$
  and $n$ in $\omega$. Clearly, $(\mathcal{R},\mathcal{S},\mik,\mik^o, r,s)$ is a Ramsey space.
  Moreover it is easy to observe that axiom \textbf{A.1} is satisfied and $\mathcal{S}$ is a closed subset of $\mathcal{AS}^\omega$.

  We define the binary relation $\mik_\mathrm{fin}$ on $\mathcal{AS}$ as follows. For $x,y$ in
  $\mathcal{AS}$ we set $x\mik_\mathrm{fin}y$ if either both $x,y$ are the empty sequence or both $x,y$ are not the empty sequence and $[x]_\Lambda\subseteq[y]_\Lambda$. We define the binary relation $\mik_\mathrm{fin}^o$ as follows. For $x$ in $\mathcal{AS}$ and $a$ in $\mathcal{AR}$ we set $a\mik_\mathrm{fin}^ox$ if either both $a$ and $x$ are the empty sequence or both $a,x$ are not the empty sequence and $\cup a\in[x]_\Lambda$, where by $\cup a$ we denote the union of the elements of $a$.

  Under these definitions, the validity of axioms \textbf{A.2} and \textbf{A.3} is straightforward. Axiom \textbf{A.4} follows by Theorem \ref{tree_Hales_Jewett_2}.
  Indeed, let $a$ in $\mathcal{AR}$ and $Y$ in $\mathcal{S}$ with $[a,Y]\neq\emptyset$.
  Also let $\mathcal{O}$ be a subset of $\mathcal{AR}_{l+1}$, where $l=|a|$.
  Moreover, let $a=(f_i)_{i=0}^{|a|-1}$ and  $\ell=\tp(f_{|a|-1})$ if $a$ is not the empty sequence and $\ell=0$ otherwise. Finally, let $Y=(g_i)_{i=0}^\infty$. Observe that $Y'=(f_i)_{i=\mathrm{depth}_Y(a)}^\infty$ is an $\ell$-subspace.
  Moreover, for every $b$ in $r_{l+1}[a,Y]$ there exists unique $f_b$ in $[Y']_\Lambda$ such that $b=a^\con(f_b)$ and vice versa, for every $f$ in $[Y']_\Lambda$ we have that $a^\con(f)$ belongs to $r_{l+1}[a,Y]$. We define the following coloring $c:\w(\Lambda,\mathbf{T})\to\{0,1\}$.
  First we define $c$ on $[Y']_\Lambda$. For every $f$ in $[Y']_\Lambda$ we set
  \[c(f)=1\;\text{if and only if}\;a^\con(f)\in\mathcal{O}.\]
  We arbitrarily extend $c$ on $\w(\Lambda,\mathbf{T})$. Applying Theorem \ref{tree_Hales_Jewett_2} we obtain a further subspace $X'$ of $Y'$ and $j$ in $\{0,1\}$ such that $c(A)=j$ for all $A$ in $[X']_\Lambda$. Modulo passing to a further subspace, we may assume that $X'$ is a $(\mathrm{depth}_Y(a),\ell)$-subspace. Then $X=r\big(Y,\mathrm{depth}_Y(a)\big)^\con X'$ is a subspace and, in particular, belongs to $[\mathrm{depth}_Y(a),Y]$. Let $\mathcal{O}_0=\mathcal{O}^c$ and $\mathcal{O}_1=\mathcal{O}$.
  Observe that for every $b$ in $r_{l+1}[a,X]$ we have that $f_b\in[X']_\Lambda$ and therefore, by the choice of $X'$ and  the definition of the coloring $c$ we have that $b\in\mathcal{O}_j$. That is, $r_{l+1}[a,X]\subseteq\mathcal{O}_j$.

  By the above, we have that the assumptions of Theorem \ref{abstract_Ramsey} are satisfied.
  The result follows clearly by Theorem \ref{abstract_Ramsey} in this setting.
\end{proof}

\section{Consequences}\label{section_consequences}
Our first consequence of Theorem \ref{tree_Hales_Jewett_infinite} concerns words of infinite support. To state it we need some pieces of notation. Let $\Lambda$ be a finite alphabet and $\mathbf{T}$ a vector tree of infinite height. Also let $\{v_\ttt:\ttt\in\mathbf{T}\}$ be a collection of symbols disjoint from $\Lambda$. By $W^\infty(\Lambda,\mathbf{T})$ we denote the set of all maps from $\otimes\mathbf{T}$ into $\Lambda$. We endow the set $\Lambda$ with the discrete topology and the
set $\w^\infty(\Lambda,\mathbf{T})$ with the product topology. A subspace $Q$ of $W^\infty(\Lambda,\mathbf{T})$ is a map from $\otimes\mathbf{T}$ into
$\Lambda\cup\{v_\ttt:\ttt\in\otimes\mathbf{T}\}$ such that there exists a vector subset $\mathbf{D}$ of $\mathbf{T}$ satisfying the following.
\begin{enumerate}
  \item[(i)] $\mathbf{D}$ is dense 
  in $\mathbf{T}$.
  \item[(ii)] For every $\ttt$ in $\otimes\mathbf{T}$, we have that $v_\ttt$ belongs to the range of $Q$ if and only if $\ttt$ belongs to $\otimes\mathbf{D}$.
  \item[(iii)] For every $\ttt$ in $\otimes\mathbf{D}$, the set $Q^{-1}(\{v_\ttt\})$ is finite having a minimum which is equal to $\ttt$.
  \item[(iv)] For every $\ttt$ and $\ttt'$ in $\otimes\mathbf{D}$ with $\ell_{\otimes\mathbf{T}}(\ttt)=\ell_{\otimes\mathbf{T}}(\ttt')$, we have that \[L_{\otimes\mathbf{T}}\big(Q^{-1}(v_\ttt)\big)=L_{\otimes\mathbf{T}}\big(Q^{-1}(v_{\ttt'})\big).\]
  \item[(v)] For every $\ttt$ and $\ttt'$ in $\otimes\mathbf{D}$ with $\ell_{\otimes\mathbf{T}}(\ttt)<\ell_{\otimes\mathbf{T}}(\ttt')$, we have that
      \[\max L_{\otimes\mathbf{T}}\big(Q^{-1}(v_\ttt)\big)< \min L_{\otimes\mathbf{T}}\big(Q^{-1}(v_{\ttt'})\big).\]
\end{enumerate}
For every subspace $Q$ of $\w^\infty(\Lambda,\mathbf{T})$, we denote the corresponding vector subset $\mathbf{D}$ by $\mathbf{D}(Q)$. Let us observe that $Q$ is a subspace of $\w^\infty(\Lambda,\mathbf{T})$ if and only if there exists a subspace $X=(f_n)_{n=0}^\infty$ of $\w(\Lambda,\mathbf{T})$ such that $Q=\cup X$, where $\cup X=\bigcup_{n=0}^\infty f_n$. Although for a subspace $Q$ of $\w^\infty(\Lambda,\mathbf{T})$ the subspace $X$ of $\w(\Lambda,\mathbf{T})$ satisfying $Q=\cup X$ is not unique, we have that for every such $X$ the set $\big\{\bigcup_{n=0}^\infty g_n:(g_n)_{n=0}^\infty\in[X]^{[\infty]}_\Lambda\big\}$ is the same.
For a subspace $Q$ of $\w^\infty(\Lambda,\mathbf{T})$, we define
\[[Q]_\Lambda=\Big\{\bigcup_{n=0}^\infty g_n:(g_n)_{n=0}^\infty\in[X]^{[\infty]}_\Lambda\Big\},\]
where $X$ is a subspace of $\w(\Lambda,\mathbf{T})$ satisfying $Q=\cup X$. Finally,
if $Q$ and $Q'$ are subspaces of $\w^\infty(\Lambda,\mathbf{T})$, we say that $Q'$ is a further subspace of $Q$, we write $Q'\mik Q$, if $[Q']_\Lambda$ is a subset of $[Q]_\Lambda$. By Theorem
\ref{tree_Hales_Jewett_infinite} we have the following immediate consequence.
\begin{cor}
  \label{tree_Hales_Jewett_infinite_cor}
  Let $\Lambda$ be a finite alphabet and $\mathbf{T}$ a vector tree of infinite height. Then for every finite Souslin measurable coloring of the set $\w^\infty(\Lambda,\mathbf{T})$ and every subspace $Q$ of $\w^\infty(\Lambda,\mathbf{T})$ there exists a subspace $Q'$ of $\w^\infty(\Lambda,\mathbf{T})$ with $Q'\mik Q$ such that the set $[Q']_\Lambda$ is monochromatic.
\end{cor}

Our second consequence of Theorem \ref{tree_Hales_Jewett_infinite} concerns maps defined on a
dense 
vector subset of $\mathbf{T}$.
This is the result that will be needed for the proof of our main result. In order to state it, we need some additional
notation. Let $\mathbf{D}$ be a dense 
vector subset of $\mathbf{T}$. We set $\w^\infty(\Lambda,\mathbf{T},\mathbf{D})$ to be the set of all maps from $\otimes\mathbf{D}$ into $\Lambda$.
The elements of $\w^\infty(\Lambda,\mathbf{T},\mathbf{D})$ can be viewed as restrictions of the elements in $\w^\infty(\Lambda,\mathbf{T},)$ on $\otimes\mathbf{D}$.
We endow $\Lambda$ with the discrete topology and $\w^\infty(\Lambda,\mathbf{T},\mathbf{D})$ with the product topology.
A map $F$ from
$\otimes\mathbf{D}$ into $\Lambda\cup\{v_\ttt:\ttt\in\otimes\mathbf{T}\}$ is called a subspace of $\w^\infty(\Lambda,\mathbf{T},\mathbf{D})$ if there exists a subspace $Q$ of $\w^\infty(\Lambda,\mathbf{T})$ such that $F=Q\upharpoonright\otimes\mathbf{D}$ and $\mathbf{D}(Q)$ is a vector subset of $\mathbf{D}$.
Observe that for every subspace $F$ of
$\w^\infty(\Lambda,\mathbf{T},\mathbf{D})$ and subspaces $Q,Q'$ of $\w^\infty(\Lambda,\mathbf{T},\mathbf{D})$ such that $F=Q\upharpoonright\otimes\mathbf{D}=Q'\upharpoonright\otimes\mathbf{D}$ and both $\mathbf{D}(Q), \mathbf{D}(Q')$ are vector subsets of $\mathbf{D}$ we have that $\mathbf{D}(Q)=\mathbf{D}(Q')$ and
\[\big\{A\upharpoonright\otimes\mathbf{D}:\;A\in[Q]_\Lambda\big\}=\big\{A\upharpoonright\otimes\mathbf{D}:\;A\in[Q']_\Lambda\big\}.\]
For every subspace $F$ of
$\w^\infty(\Lambda,\mathbf{T},\mathbf{D})$ we define $\mathbf{D}(F)=\mathbf{D}(Q)$ and
\[[F]_\Lambda=\big\{A\upharpoonright\otimes\mathbf{D}:\;A\in[Q]_\Lambda\big\},\]
where $Q$ is a subspace of $\w^\infty(\Lambda,\mathbf{T},\mathbf{D})$ such that $F=Q\upharpoonright\otimes\mathbf{D}$ and $\mathbf{D}(Q)$ is a vector subset of $\mathbf{D}$.
Finally,
if $F$ and $F'$ are subspaces of $\w^\infty(\Lambda,\mathbf{T},\mathbf{D})$, we say that $F'$ is a further subspace of $F$, we write $F'\mik F$, if $[F']_\Lambda$ is a subset of $[F]_\Lambda$. By Corollary
\ref{tree_Hales_Jewett_infinite_cor}, we have the following immediate consequence.
\begin{cor}
  \label{tree_Hales_Jewett_infinite_dominating}
  Let $\Lambda$ be a finite alphabet and $\mathbf{T}$ a vector tree of infinite height.
  Also let $\mathbf{D}$ be a dominating vector subset of $\mathbf{T}$.
  Then for every finite Souslin measurable coloring of the set $\w^\infty(\Lambda,\mathbf{T},\mathbf{D})$ and every subspace $F$ of $\w^\infty(\Lambda,\mathbf{T},\mathbf{D})$ there exists a subspace $F'$ of $\w^\infty(\Lambda,\mathbf{T},\mathbf{D})$ with $F'\mik F$ such that the set $[F']_\Lambda$ is monochromatic.
\end{cor}

\section{Disjoint union theorem for trees} \label{section_main_res}
We have developed all the tools needed for the proof of the main result of the present work, that is, Theorem \ref{disjoint_Union_tree}. Before we proceed to its proof, let us recall the relevant notation.
Let $\mathbf{T}$ be a vector tree of infinite height.
Recall that \[\mathcal{U}(\mathbf{T})=\{U\subseteq\otimes\mathbf{T}:\;U\;\text{has a minimum}\}.\]
As we have already mentioned in the introduction the collection $\mathcal{U}(\mathbf{T})$ can be viewed as a subset of the set $\{0,1\}^{\otimes\mathbf{T}}$ of all
functions from $\otimes\mathbf{T}$ into $\{0,1\}$. We endow the set $\{0,1\}^{\otimes\mathbf{T}}$
with the product topology of the discrete topology on $\{0,1\}$. Then
$\mathcal{U}(\mathbf{T})$ forms a closed subset of $\{0,1\}^{\otimes\mathbf{T}}$.

Let $\mathbf{D}$ be a vector subset of $\mathbf{T}$. Recall that a $\mathbf{D}$-subspace of $\mathcal{U}(\mathbf{T})$ is a family $\mathbf{U}=(U_\ttt)_{\ttt\in\otimes\mathbf{D}}$ consisting of pairwise disjoint elements from $\mathcal{U}(\mathbf{T})$ such that $\min U_\ttt=\ttt$ for all $\ttt$ in $\otimes\mathbf{D}$, while $\mathbf{U}$ is a subspace of $\mathcal{U}(\mathbf{T})$ if it is a $\mathbf{D}$-subspace for some vector subset $\mathbf{D}$ of $\mathbf{T}$, which we denote by $\mathbf{D}(\mathbf{U})$.
For a subspace $\mathbf{U}=(U_\ttt)_{\ttt\in\otimes\mathbf{D}(\mathbf{U})}$
its span is define as
\[[\mathbf{U}]=\Big\{\bigcup_{\ttt\in\Gamma}U_\ttt:\;\Gamma\subseteq\otimes\mathbf{D}(\mathbf{U})\Big\}\cap\mathcal{U}(\mathbf{T}).\]
Finally, recall that if $\mathbf{U}$ and $\mathbf{U}'$ are two subspaces of $\mathcal{U}(\mathbf{T})$, we say that $\mathbf{U}'$ is a further subspace of $\mathbf{U}$, we write $\mathbf{U}'\mik\mathbf{U}$, if $[\mathbf{U}']$ is a subset of $[\mathbf{U}]$. The latter, in particular, implies that $\mathbf{D}(\mathbf{U}')$ is a vector subset of $\mathbf{D}(\mathbf{U})$.

The main tool for the proof of Theorem \ref{disjoint_Union_tree} is the following lemma. To state it we need some additional notation. Let $\mathbf{D}$ be dense 
vector subset of $\mathbf{T}$, $n$ a non-negative integer and $\mathbf{U}=(U_\ttt)_{\ttt\in\otimes\mathbf{D}}$ a $\mathbf{D}$-subspace of $\mathcal{U}(\mathbf{T})$.
We set $\mathbf{U}\upharpoonright n=(U_\ttt)_{\ttt\in\otimes(\mathbf{D}\upharpoonright n)}$.
\begin{lem}
  \label{lem_fix_one_min}
  Let $\mathbf{T}$ be a vector tree of infinite height, $r$ a positive integer and $c:\mathcal{U}(\mathbf{T})\to\{1,...,r\}$ a Souslin measurable coloring. Also let
  $\mathbf{D}$ be a dense 
  vector subset of $\mathbf{T}$ and $\mathbf{U}$ a $\mathbf{D}$-subspace of $\mathcal{U}(\mathbf{T})$. Finally, let  $n$ be a non-negative integer and $\ttt$ in $\otimes\mathbf{D}(n)$.
  Then there exist a dense 
  vector subset $\mathbf{D}'$ of $\mathbf{D}$
  and a $\mathbf{D}'$-subspace $\mathbf{U}'$ of $\mathcal{U}(\mathbf{T})$ satisfying the following.
  \begin{enumerate}
    \item[(i)] $\mathbf{D}'\upharpoonright n+1=\mathbf{D}\upharpoonright n+1$.
    \item[(ii)] $\mathbf{U}'\mik\mathbf{U}$ and $\mathbf{U}'\upharpoonright n=
  \mathbf{U}\upharpoonright n$
    \item[(iii)] $c(U)=c(U')$ for all $U,U'$ in $[\mathbf{U}']$ with $\min U=\min U'=\ttt$.
  \end{enumerate}
\end{lem}
\begin{proof}
  Let $\mathbf{T}=(T_1,...,T_d)$ and $\mathbf{D}=(D_1,...,D_d)$. Also let $\ttt=(t_1,...,t_d)$ and $\mathbf{U}=(U_\sss)_{\sss\in\otimes\mathbf{D}}$. For every $1\mik i\mik d$ we set $T'_i=\{t'\in T_i:t_i\mik_{T_i}t' \}$ and $D'_i=(D_i\cap T'_i)\setminus\{t_i\}$.
  We also set $\mathbf{T}'=(T'_1,...,T'_d)$ and $\mathbf{D}'=(D'_1,...,D'_d)$. Finally, we set
  \[\mathcal{A}=\{U\in\mathcal{U}(\mathbf{T}'):\;\min U=\ttt\}\cap[\mathbf{U}].\]
  Observe that $\mathcal{A}=\{U\in[\mathbf{U}]:\;\min U=\ttt\}$ and, in particular, that $\mathcal{A}$ forms a closed subset of $\mathcal{U}(\mathbf{T})$. Thus the restriction of $c$ on $\mathcal{A}$ is Souslin measurable. We set $\Lambda=\{0,1\}$ and we define a map $Q:\w^\infty(\Lambda,\mathbf{T}',\mathbf{D}')\to\mathcal{A}$ as follows. For every $f$ in $\w^\infty(\Lambda,\mathbf{T}',\mathbf{D}')$, we set
  \[Q(f)=U_\ttt\cup\Big(\bigcup_{\sss\in f^{-1}\big(\{1\}\big)}U_\sss\Big).\]
  It is easy to see that $Q$ is 1-1, onto and continuous. Let $c^*:\w^\infty(\Lambda,\mathbf{T}',\mathbf{D}')\to\{1,...,r\}$ be defined by $c^*(f)=c(Q(f))$
  for all $f\in\w^\infty(\Lambda,\mathbf{T}',\mathbf{D}')$. Then we have that $c^*$ is Souslin measurable. By Corollary \ref{tree_Hales_Jewett_infinite_dominating}, there exists a subspace $F$ of $\w^\infty(\Lambda,\mathbf{T}',\mathbf{D}')$ such that the set $[F]_\Lambda$ is $c^*$-monochromatic.
  We set \[U'_\ttt=U_\ttt\cup\Big(\bigcup_{\sss\in Q^{-1}(\{1\})}U_\sss\Big)\] and for every $\sss\in\otimes\mathbf{D}(F)$ we set
  \[U'_\sss=\bigcup_{\sss'\in F^{-1}(\{v_\sss\})}U_{\sss'}.\]
  By the definition of the coloring $c^*$ and the choice of $F$, we have that the set
  \[\mathcal{B}=\big\{U\in\big[(U'_\sss)_{\sss\in\{\ttt\}\cup\otimes\mathbf{D}(F)}\big]:\;\min U=\ttt\big\}\]
  is $c$-monochromatic. Let $\mathbf{D}(F)=(D^*_1,...,D^*_d)$. Pick any dominating vector subset $\mathbf{D}'=(D'_1,...,D'_d)$ of $\mathbf{D}$ such that $\mathbf{D}'\upharpoonright n+1=\mathbf{D}\upharpoonright n+1$ and $D'_i\cap T_i\subseteq D^*_i$ for all $1\mik i\mik d$.
  For every $\sss$ in $\otimes\mathbf{D}'\setminus\otimes\mathbf{D}(F)$ with $\sss\neq\ttt$ we set $U'_\sss=U_\sss$. Setting $\mathbf{U}'=(U'_\sss)_{\sss\in\otimes\mathbf{D}'}$, it is easy to check that $\mathbf{D}'$ and $\mathbf{U}'$ are as desired.
\end{proof}

Since for every non-negative integer $n$ and dense 
vector subset $\mathbf{D}$ of a vector tree $\mathbf{T}$ of infinite height the set $\otimes\mathbf{D}(n)$ is finite, iterating Lemma \ref{lem_fix_one_min}, we obtain the following.

\begin{lem}
  \label{lem_fix_min_full_level}
  Let $\mathbf{T}$ be a vector tree of infinite height, $r$ a positive integer and $c:\mathcal{U}(\mathbf{T})\to\{1,...,r\}$ a Souslin measurable coloring. Also let
  $\mathbf{D}$ be a dense 
  vector subset of $\mathbf{T}$ and $\mathbf{U}$ a $\mathbf{D}$-subspace of $\mathcal{U}(\mathbf{T})$. Finally, let  $n$ be a non-negative integer.
  Then there exist a dense 
  vector subset $\mathbf{D}'$ of $\mathbf{D}$
  and a $\mathbf{D}'$-subspace $\mathbf{U}'$ of $\mathcal{U}(\mathbf{T})$ satisfying the following.
  \begin{enumerate}
    \item[(i)] $\mathbf{D}'\upharpoonright n+1=\mathbf{D}\upharpoonright n+1$.
    \item[(ii)] $\mathbf{U}'\mik\mathbf{U}$ and $\mathbf{U}'\upharpoonright n=
  \mathbf{U}\upharpoonright n$
    \item[(iii)] $c(U)=c(U')$ for all $U,U'$ in $[\mathbf{U}']$ with $\min U=\min U'$ in $\otimes\mathbf{D}'(n)$.
  \end{enumerate}
\end{lem}
The next lemma achieves a canonicalization of the coloring with respect to the minimum.
\begin{lem}
  \label{lem_fix_min_full_tree}
  Let $\mathbf{T}$ be a vector tree of infinite height, $r$ a positive integer and $c:\mathcal{U}(\mathbf{T})\to\{1,...,r\}$ a Souslin measurable coloring. Also let
  $\mathbf{D}$ be a dense 
  vector subset of $\mathbf{T}$ and $\mathbf{U}$ a $\mathbf{D}$-subspace of $\mathcal{U}(\mathbf{T})$.
  Then there exists a subspace $\mathbf{U}'$ of $\mathcal{U}(\mathbf{T})$ with $\mathbf{U}'\mik\mathbf{U}$ such that $\mathbf{D}(\mathbf{U}')$ is a dense 
  vector subset of $\mathbf{T}$
  and
  $c(U)=c(U')$ for all $U,U'$ in $[\mathbf{U}']$ with $\min U=\min U'$.
\end{lem}
\begin{proof}
  Applying Lemma \ref{lem_fix_min_full_tree}, we inductively construct a sequence $(\mathbf{D}_n)_{n=0}^\infty$ of dense 
  vector subsets of $\mathbf{T}$ with $\mathbf{D}_0=\mathbf{D}$ and a sequence $(\mathbf{U}_n)_{n=0}^\infty$ of subspaces of $\mathcal{U}(\mathbf{T})$ with $\mathbf{U}_0=\mathbf{U}$ satisfying the following for every non-negative integer $n$.
  \begin{enumerate}
    \item[(i)] $\mathbf{D}_{n+1}$ is a vector subset of $\mathbf{D}_n$ with $\mathbf{D}_{n+1}\upharpoonright n+1=\mathbf{D}_n\upharpoonright n+1$.
    \item[(ii)] $\mathbf{U}_{n+1}\mik\mathbf{U}_n$ and $\mathbf{U}_{n+1}\upharpoonright n=
  \mathbf{U}_n\upharpoonright n$
    \item[(iii)] $c(U)=c(U')$ for all $U,U'$ in $[\mathbf{U}_{n+1}]$ with $\min U=\min U'$ in $\otimes\mathbf{D}_{n+1}(n)$.
  \end{enumerate}
  We set $\mathbf{D}'=\bigcup_{n=0}^\infty\mathbf{D}_n(n)$. Moreover, for every non-negative integer $n$ and $\ttt$ in $\otimes\mathbf{D}'(n)$, if $\mathbf{U}_n=(U^n_\sss)_{\sss\in\otimes\mathbf{D}_n}$, we set $U'_\ttt=U^n_\ttt$.
  Setting $\mathbf{U}'=(U'_\sss)_{\sss\in\otimes\mathbf{D}'}$, it is easy to check that $\mathbf{U}'$ is as desired.
\end{proof}

We are ready for the proof of Theorem \ref{disjoint_Union_tree}.
\begin{proof}
  [Proof of Theorem \ref{disjoint_Union_tree}]
  We define a coloring $c:\mathcal{U}(\mathbf{T})\to\{0,1\}$ setting for every $U$ in $\mathcal{U}(\mathbf{T})$
  \[
c(U)=\left\{ \begin{array} {l} 0\;\;\;\;\text{if}\;U\not\in\mathcal{P},\\
1\;\;\;\;\text{if}\;U\in\mathcal{P}.\end{array}  \right.
\]
Since $\mathcal{P}$ is a Souslin measurable subset of $\mathcal{U}(\mathbf{T})$,
we have that $c$ is a Souslin measurable coloring of $\mathcal{U}(\mathbf{T})$.
Applying Lemma \ref{lem_fix_min_full_tree}, we obtain a subspace $\mathbf{U}''$ of
$\mathcal{U}(\mathbf{T})$ with $\mathbf{U}''\mik\mathbf{U}$ such that
$\mathbf{D}(\mathbf{U}'')$ is a dense 
vector subset of $\mathbf{T}$
  and
  \begin{equation}
    \label{eq03}
    c(U)=c(U')
  \end{equation}
  for all $U,U'$ in $[\mathbf{U}'']$ with $\min U=\min U'$. For every $\sss$ in $\otimes\mathbf{D}(\mathbf{U}'')$ we pick an element $U_\sss$ from $[\mathbf{U}'']$
  with $\min U_\sss=\sss$. We define a subset $\mathcal{P}_*$ of $\otimes\mathbf{D}(U'')$ as follows:
  \[\mathcal{P}_*=\{\sss\in\otimes\mathbf{D}(\mathbf{U''}):\;c(U_\sss)=1\}.\]
  By the definition of the coloring $c$ and \eqref{eq03} we have the following property.
  \begin{enumerate}
    \item[($\mathrm{P}$)] For every $U$ in $[\mathbf{U}'']$ we have that $U\in\mathcal{P}$ if and only if $\min U\in\mathcal{P}_*$.
  \end{enumerate}
  Applying Halpern--L\"auchli Theorem, that is, Theorem \ref{Halpern_Lauchli},
  we obtain
  a vector subset $\mathbf{D}'$ of $\mathbf{D}(U'')$ such that either
  \begin{enumerate}
    \item[(i$'$)] $\otimes\mathbf{D}'$ is a subset of $\mathcal{P}_*$ and $\mathbf{D}'$ is a dense 
        vector subset of $\mathbf{T}$, or
    \item[(ii$'$)] $\otimes\mathbf{D}'$ is a subset of $\mathcal{P}_*^c$ and $\mathbf{D}'$ is
        a $\ttt$-dense 
        vector subset $\mathbf{D}'$ of $\mathbf{T}$ for some $\ttt$ in $\otimes\mathbf{T}$.
  \end{enumerate}
  Let $\mathbf{U}'=(U_\sss)_{\sss\in\otimes\mathbf{D}'}$. Then $\mathbf{U}'\mik\mathbf{U}''\mik\mathbf{U}$ and $\mathbf{D}(\mathbf{U}')=\mathbf{D}'$.
  Thus invoking property ($\mathrm{P}$) we have that either
  \begin{enumerate}
    \item[(i)]  $[\mathbf{U}']$ is a subset of $\mathcal{P}$ and $\mathbf{D}(\mathbf{U}')$ is a dominating vector subset of $\mathbf{T}$, or
    \item[(ii)] $[\mathbf{U}']$ is a subset of $\mathcal{P}^c$ and $\mathbf{D}(\mathbf{U}')$ is a $\ttt$-dominating vector subset of $\mathbf{T}$ for some $\ttt$ in $\otimes\mathbf{T}$.
  \end{enumerate}
  The proof of the theorem is complete.
\end{proof}

\section{Concluding Remarks}

It is natural to ask whether a multidimensional version of Theorem \ref{disjoint_Union_tree} holds true. Of course that would imply a multidimensional version of Halpern-L\"{a}uchli  Theorem, which is true only in its strong tree version, i.e. Milliken's tree Theorem \cite{Mi1, Mi2}. However,
one can easily see that this fails too. In particular, if $T$ is a tree and $\mathcal{A}$ is the collection of all pair $(U_1,U_2)$ of disjoint subset of $T$ having a minimum with $\min U_1\mik_T\min U_2$, then we can define a coloring $c:\mathcal{A}\to\{0,1\}$ setting $c\big((U_1,U_2)\big)$ to be the cardinality of the set $U_1\cap T(\ell_T(\min U_2))$ modulo $2$.
It is easy to see that the span of every subspace indexed by a strong subtree achieves both colors. Actually, obvious modifications of this coloring result the non-existence of a Ramsey degree.

However, a positive result towards this direction can be achieved in the case of finite homogeneous trees and a modified notion of subspaces, namely, subspaces indexed by a skew subtree. This is a subject of a forthcoming paper. The case of infinite trees remains still open.



\begin{thebibliography}{99}
\bibitem{AFK} S. A. Argyros, V. Felouzis and V. Kanellopoulos,
\textit{A proof of Halpern-L\"{a}uchli partition theorem},
European J. Combin. 23 (2002), no. 1, 1–10.

\bibitem{C} T. J. Carlson, \textit{Some unifying principles in Ramsey Theory}, Discrete Math., 68 (1988), 117--169.

\bibitem{CS} T. J. Carlson and S. G. Simpson, \textit{A dual form of Ramsey's theorem}, Adv. Math., 53 (1984), 265--290.

\bibitem{CS2} T. J. Carlson and S. G. Simpson, textit{Topological Ramsey theory. In
Mathematics of Ramsey theory}, volume 5 of Algorithms Combin., pages
172--183. Springer, Berlin, 1990.

\bibitem{GRS}  R.L. Graham, B.L. Rothschild, and J. H. Spencer, \textit{ Ramsey theory. Second edition.} Wiley-Interscience Series in Discrete Mathematics and Optimization. A Wiley-Interscience Publication. John Wiley and Sons, Inc., New York, 1990. xii+196 pp.

\bibitem{FK}  H. Furstenberg and Y. Katznelson, \textit{Idempotents in compact semigroups and Ramsey Theory}, Israel J. Math. 68 (1989), 257--270.

\bibitem{HJ} A. H. Hales and R. I. Jewett, \textit{Regularity and positional games}, Trans. Amer. Math. Soc., 106 (1963), 222--229.

\bibitem{HL} J. D. Halpern and H. L\"{a}uchli, \textit{A partition theorem}, Trans. Amer. Math. Soc., 124 (1966), 360--367.

\bibitem{K}  N. Karagiannis, \textit{A combinatorial proof of an infinite version of the Hales-Jewett theorem}, J. Comb. 4 (2013), no. 2, 273--291.

\bibitem{Mi1} K. Milliken, \textit{A Ramsey theorem for trees}, J. Comb. Theory Ser. A, 26 (1979), 215-237.

\bibitem{Mi2} K. Milliken, \textit{A partition theorem for the infinite subtrees of a tree}, Trans. Amer. Math. Soc., 263 (1981), 137-148.

\bibitem{To} S. Todorcevic, \textit{Introduction to Ramsey Spaces}, Annals Math. Studies, No. 174, Princeton Univ. Press, 2010.

\end{thebibliography}
\end{document}